\documentclass[11pt,a4paper]{article}

\usepackage[utf8]{inputenc}
\usepackage[T1]{fontenc}
\usepackage{amsmath,amsthm,amsopn,amsfonts,amssymb,dsfont,color}
\usepackage{mathrsfs}
\usepackage{tikz}

\usepackage{xcolor}
\usepackage{hyperref}
\usepackage{dsfont}
\usepackage{stmaryrd}
\usepackage{a4,a4wide}

\def\R {\mathbb R}
\def\T{\mathcal T}
\def\E{\mathcal E}
\def\P{\mathcal P}
\def\M{\mathcal M}
\def\H{\mathcal H}
\def\D{\mathcal D}

\def\ds{\displaystyle}
\def\pa{\partial}
\def\Linf{L^\infty}
\def\m{M_0}
\def\moy{\langle f \rangle}
\def\vinf{v^\infty}
\def\uinf{u^\infty}

\def\TO{\T_\Omega}
\def\To{\T_\omega}
\def\Ke{K^*}
\def\ts{\tau_\sigma}
\def\tse{\tau_{\sigma^*}}

\def\mKe{m_{K^*}}
\def\uKe{u_{K^*}}
\def\vKe{v_{K^*}}

\def\vK{v_K}
\def\phiKe{\varphi_{K^*}}
\def\psiKe{\psi_{K^*}}
\def\psiLe{\psi_{L^*}}
\def\phiK{\varphi_K}

\def\HPhi{{\mathcal H}_\Phi}
\def\DPhi{{\mathcal D}_\Phi}
\def\Hd{{\mathcal H}_2}
\def\Dd{{\mathcal D}_2}

\def\R{\mathbb R}

\newcommand{\mo}{m_\omega}
\newcommand{\mO}{m_\Omega}

\newtheorem{theorem}{\textbf{Theorem}}[section]
\newtheorem{lemma}[theorem]{\textbf{Lemma}}
\newtheorem{proposition}[theorem]{\textbf{Proposition}}

\theoremstyle{remark}
\newtheorem{remark}[theorem]{\textbf{Remark}}

\numberwithin{equation}{section}

\title{\bf Long time behavior of the field-road diffusion model: an entropy method and a finite volume scheme}
\date{}

\begin{document}
\maketitle

\begin{center}
{\large \bf Matthieu Alfaro\footnote{Univ. Rouen Normandie, CNRS, LMRS UMR 6085, F-76000 Rouen, France.} and Claire Chainais-Hillairet\footnote{Univ. Lille, CNRS, Inria, UMR 8524 - Laboratoire Paul Painlevé, F-59000 Lille, France.}}
\end{center}

\vspace{15pt}
\tableofcontents

\vspace{10pt}
\begin{abstract} 
We consider the so-called  field-road  diffusion model in a bounded domain, consisting of  two parabolic PDEs posed on sets of different dimensions (a {\it  field} and a {\it road} in a population dynamics context) and coupled through exchange terms on the road, which makes its analysis quite involved. We propose a TPFA finite volume scheme. In both the continuous and  the discrete settings, we prove the
exponential decay of an entropy, and thus the long time  convergence to  the stationary state selected by the total mass of the initial data.  To deal with the problem of different dimensions, we artificially \lq\lq thicken'' the road and, then, establish a rather unconventional Poincaré-Wirtinger inequality. Numerical simulations confirm and complete the analysis,  and raise new issues.
\\


\end{abstract}

\vspace{3pt}

\section{Introduction}\label{s:intro}

Phenomena of spatial spread are highly relevant to understand biological invasions, spreads of emergent diseases, as well as spatial shifts in distributions in the context of climate change. Let us refer, among many others, to the seminal books by Shigesada and Kawasaki \cite{Shi-Kaw-97}, by  Murray \cite{Murray1, Murray2}. More recently, there has been a growing recognition of the importance of {\it fast diffusion channels} on biological invasions: for instance, an accidental transportation via human activities of some individuals towards northern and eastern France may be the cause of accelerated propagation of the pine
processionary moth \cite{Rob-et-al-12}. In Canada, some GPS data revealed that wolves travel faster along seismic lines (i.e. narrow strips cleared for energy exploration), thus increasing their chances to meet a prey \cite{MacKen-et-al-12}. It is also acknowledged that fast diffusion channels (roads, airlines, etc.) play a central role in the propagation of epidemics. As is well known, the spread of the black plague, which killed about a third of the European population in the 14th century, was favoured by the trade routes, especially the Silk Road, see  \cite{PesteNoireRoutesSoie}. More recently, some evidences of the the radiation of the COVID epidemic along highways and transportation infrastructures were found \cite{Gat-et-al-20}.

The so-called {\em field-road} model was introduced  by Berestycki, Roquejoffre and Rossi \cite{Ber-Roq-Ros-13-1} in order to describe such spread of diseases or invasive species in presence of networks with fast propagation. It is set on an unbounded domain. We will recall it hereafter and review the main established mathematical results. The current work is devoted to the theoretical and numerical analysis of a {\it purely diffusive} field-road model set on a bounded domain. We focus on the analysis of its long time behavior. 

\subsection{The continuous field-road diffusion model}\label{ss:continuous}

The {\em field-road} model introduced by Berestycki, Roquejoffre and Rossi \cite{Ber-Roq-Ros-13-1} writes as
\begin{equation}\label{syst-BRR}
\left\{
\begin{array}{ll}
	\partial_{t} v = d \Delta v + f(v),
&\quad t>0, \; x \in \R^{N-1}, \; y>0, \vspace{4pt}\\
	- d\, \partial_{y} v|_{y=0} = \mu u - \nu v|_{y=0},
&\quad t>0, \; x \in \R^{N-1}, \vspace{4pt}\\
	\partial_{t} u = D \Delta u + \nu v|_{y=0} - \mu u,
&\quad t>0, \; x \in \R^{N-1}.
\end{array}
\right.
\end{equation}
 The mathematical problem then amounts to describing survival and propagation in a non-standard physical space: the geographical domain consists in the half-space (the \lq\lq field'') $x\in \R^{N-1}$, $y>0$, bordered by the hyperplane (the \lq\lq road'') $x\in \R^{N-1}$, $y=0$. In the field, individuals diffuse with coefficient $d>0$ and their density is given  by $v=v(t,x,y)$. In particular $\Delta v$ has to be understood as $\Delta_x v+\partial_{yy}v$. On the road, individuals typically diffuse faster ($D>d$) and their density is given by $u=u(t,x)$. In particular $\Delta u$ has to be understood as $\Delta_x u$. The exchanges of population between the road and the field are described by the second equation in system \eqref{syst-BRR}, where $\mu>0$ and $\nu >0$. These boundary conditions, and the zeroth-order term on the road, link the field and the road equations and are the core of the model. 
 
In a series of works \cite{Ber-Roq-Ros-13-2, Ber-Roq-Ros-13-1, Ber-Roq-Ros-shape, Ber-Roq-Ros-tw},  Berestycki, Roquejoffre and Rossi studied the field-road system with $N=2$ and $f$ a Fisher-KPP nonlinearity. They shed light on an {\it acceleration phenomenon}: when $D>2d$, the road  enhances the global diffusion and  the spreading speed exceeds the standard Fisher-KPP invasion speed. This new feature has stimulated many
works and, since then, many related problems taking into account heterogeneities, more complex geometries, nonlocal diffusions, etc. have been studied  \cite{AC1, AC2}, \cite{Gil-Mon-Zho-15}, 
  \cite{PauthierLongRangeExchanges1, PauthierLongRangeExchanges2, PauthierLongRangeExchanges3}, \cite{Tel-16}, \cite{Ros-Tel-Val-17}, \cite{Duc-18}, \cite{BDR1, Ber-Duc-Ros-20}, \cite{Zha-21}, \cite{Bog-Gil-Tel-21}, \cite{Aff-22}. 

Very recently, the authors in \cite{Alf-Duc-Tre-23} considered the {\it purely diffusive} field-road system --- obtained by letting $f\equiv 0$ in \eqref{syst-BRR} --- as a starting point. They obtained  an explicit expression for both the fundamental solution and the solution to the associated Cauchy problem, and a
 sharp (possibly up to a logarithmic term) decay rate of the $L^\infty$ norm of the solution.

From now on, we consider the purely diffusive field-road model on a bounded domain, namely $\Omega\subset \R^N$ ($N\geq 2$)  a bounded cylinder of the form
$$
\Omega=\omega \times (0,L), \quad \omega \text{ a bounded convex and open set of } \R^{N-1},\; L>0.
$$
We still denote by  $v=v(t,x,y)$ and $ u=u(t,x)$ the densities of species respectively in the field and on the road. They are smooth solutions to the system 
\begin{equation}\label{syst}
\left\{
\begin{array}{ll}
	\partial_{t} v = d \Delta v,
&\quad t>0, \; x \in \omega, \; y\in(0,L), \vspace{5pt}\\
	- d\, \partial_{y} v|_{y=0}= \mu u - \nu v|_{y=0},
&\quad t>0, \; x \in \omega,\vspace{5pt}\\
	\partial_{t} u = D \Delta u + \nu v|_{y=0} - \mu u,
&\quad t>0, \; x \in \omega, \vspace{5pt}\\
 \frac{\partial u}{\partial n'} = 0,
&\quad t>0, \; x \in \partial \omega,\vspace{5pt}\\
 \frac{\partial v}{\partial n} = 0,
&\quad t>0, \; x \in \partial \omega, \;y\in(0,L), \text{ and } \; x\in \omega, \; y=L,\\
\end{array}
\right.
\end{equation}
 supplemented with an initial condition $(v_0,u_0)\in \Linf(\Omega)\times \Linf(\omega)$.
As in the classical model,  $d$ and  $D$ are positive diffusion coefficients, while  $\mu$ and $\nu$ are positive transfer coefficients.  For $u$ we impose the zero Neumann boundary conditions on the boundary $\partial \omega$ ($n'$ denotes the unit outward normal vector to $\partial \omega$). For $v$, we impose the zero Neumann boundary conditions on the lateral boundary $\partial \omega \times (0,L)$ and on the upper boundary $\omega \times \{L\}$ ($n$ denotes the unit outward normal vector to $\partial \Omega$). On the lower boundary $\omega \times \{0\}$, we impose the aforementioned boundary conditions linking $v$ and $u$, which is the essence of the model.

As the system \eqref{syst} is made of two diffusive equations coupled through the transfer terms, we expect convergence towards a steady-state in long time. This convergence result comes from the dissipative structure of the model. Moreover, we aim at designing a numerical scheme for \eqref{syst} that preserves such a  dissipative structure. 

\subsection{The TPFA finite volume scheme}\label{ss:TPFA}
 
In system \eqref{syst}, the diffusion processes on the road and in the field are obviously isotropic and homogeneous. Moreover, 
 we can consider \lq\lq nice'' geometries for the road and the field, so that the construction of  meshes for the domains is not a challenge (in many cases, cartesian grids would be sufficient). Therefore, 
 Two-Point Flux Approximation Finite Volume schemes seem to be adapted for the discretization of \eqref{syst}. We refer to the book by Eymard, Gallouët and Herbin \cite{Eym-Gal-Her-00}, and to references therein, for a detailed presentation of finite volume methods. In many different frameworks, these methods have proved to be well-adapted for the preservation of long time behavior of diffusive problems, see for instance \cite{Cha-Jun-Sch-16}, \cite{Cha-Her-20}.
 
In order to write a numerical scheme for the field-road model, we have to define two meshes, one for the field and one for the road, with a compatibility relation between both meshes.   This is a crucial  point to treat correctly the exchanges between the field and the road. We emphasize that the design of the scheme is driven by the will to preserve the main features of the model (mass conservation, positivity of the densities, steady-states, long-time behavior, etc.). 

\medskip
 
\noindent {\bf Meshes and notations.} Let us first consider a mesh $\M_{\Omega}$ of the field $\Omega$ made of a family of control volumes $\T_\Omega$, a family of faces (or edges) $\E_\Omega$ and a family of points $\P_\Omega$, so that $\M_{\Omega}=(\T_{\Omega}, \E_{\Omega},\P_{\Omega})$. The mesh of $\omega$ is also made of a family of control volumes, a family of edges and a family of points. It is denoted $\M_{\omega}=(\T_{\omega}, \E_{\omega},\P_{\omega})$.  We use classical notations: 
\begin{itemize}
\item $K\in\T_{\Omega}$ for a control volume, $\sigma\in\E_{\Omega}$ for an edge, $x_K\in{\mathcal P}_{\Omega}$ for an interior  point of $K$ (named as the center of $K$), 
\item  $K^*\in\T_{\omega}$ for a control volume, $\sigma^*\in\E_{\omega}$ for an edge (it can be a point when $N=2$), $x_{K^*}\in{\mathcal P}_{\omega}$ for an interior point of $K^*$.
\end{itemize}

In  $\T_{\Omega}$, we can distinguish the control volumes that have an edge on the road from the ones that are strictly included in the field, which writes $\T_{\Omega}= \T_{\Omega}^r\cup \T_{\Omega}^f$. For the edges of $\E_{\Omega}$ we can also distinguish the interior edges from the boundary edges, included in $\omega$ or included in $\partial \Omega\setminus\omega$ (considered as exterior edges), which writes  $\E_{\Omega}=\E_{\Omega}^{\rm int}\cup \E_\Omega^r\cup \E_\Omega^{\rm ext} $. For an interior edge $\sigma \in \E_{\Omega}^{\rm int}$, we may write $\sigma =K|L$ as it is an edge between the control volumes $K$ and $L$. Similarly, we can  split $\E_\omega$ into $\E_\omega= \E_\omega^{\rm int}\cup \E_\omega^{\rm ext}$ and denote each interior edge $\sigma^*\in \E_\omega^{\rm int}$ as $\sigma^*=K^*|L^*$. The main notations are illustrated on Figure~\ref{fig:mesh} in a two-dimensional case.

We assume that both meshes are admissible in the sense that they satisfy the usual orthogonality property, see \cite{Eym-Gal-Her-00}. 
This means that for each edge $\sigma= K|L$ (respectively $\sigma^*=K^*|L^*$), the line joining $x_K$ to $x_L$ (respectively $x_{K^*}$ to $x_{L^*}$) is perpendicular to $\sigma$ (respectively $\sigma^*$).  
Moreover, we assume the compatibility of the two meshes $\M_{\Omega}$ and $\M_{\omega}$: every control volume of $\T_{\omega}$ must coincide with an edge of $\E_\Omega^r$. More precisely, for all $\sigma\in\E_{\Omega}^r$, there exists a unique $K\in\T_{\Omega}^r$ such that $\sigma$ is an edge of $K$ and a unique $K^*\in\T_{\omega}$ such that $\sigma$ coincides with $K^*$. Therefore, we will use the notation $\sigma=K|K^*$ for $\sigma\in \E_{\Omega}^r$. 

\begin{figure}[htb]\begin{tabular}{ccc}\label{fig:mesh}
 \begin{tikzpicture}[scale=0.41]
 \node[above] at (17,10){$\Omega$};
\foreach \x in {10,20}{
 	\draw[line width=.8pt]  (\x+0,0)--(\x+0,10)--(\x+10,10)--(\x+10,0)--cycle;
	\draw[line width=.8pt]  (\x+0,0)--(\x+3,3)--(\x+5,0)--(\x+6.5,3.5)--(\x+10,0);
          \draw[line width=.8pt] (\x+0,5)--(\x+3,3)--(\x+6.5,3.5)--(\x+10,5)--(\x+7,7)--(\x+6.5,3.5)--(\x+3.5,6.5)--(\x+3,3);
          \draw[line width=.8pt] (\x+0,5)--(\x+3.5,6.5)--(\x+7,7)--(\x+10,10);
          \draw[line width=.8pt] (\x+0,10)--(\x+3.5,6.5)--(\x+5,10)--(\x+7,7);
          \filldraw [fill=red!30!white] (\x+0,0) -- (\x+3,3)--(\x+5,0)--(\x+6.5,3.5)--(\x+10,0)-- cycle;

          \draw[line width= 1.2pt, color = red] (\x+0,0)--(\x+10,0);
           	}
\filldraw[fill = blue!20!white] (23,3)--(23.5,6.5)--(26.5,3.5)--cycle;
\filldraw[fill = blue!20!white] (23.5,6.5)--(26.5,3.5)-- (27,7)--cycle;
\node[below] at (27,-0.2) {\color{red}$\omega$};
\node at (17,-1.5) {\phantom{$\Omega $}};
\end{tikzpicture}
&\ &
\vspace*{.2cm}
\begin{tikzpicture}[scale=0.65]
\foreach \y in {-0.2}{
\filldraw[fill = blue!20!white] (3,3+\y)--(3.5,6.5+\y)--(6.5,3.5+\y)--cycle;
\filldraw[fill = blue!20!white] (3.5,6.5+\y)--(6.5,3.5+\y)-- (7,7+\y)--cycle;
\draw[line width=1.2pt, color=blue] (3.5,6.5+\y)--(6.5,3.5+\y);
\node[above,right] at (3,3.5+\y){$K$};
\node[left,below] at (6.5,6.9+\y){$L$};
\draw[dashed, line width= 1.2pt] (4.2,\y+4.2)--(5.7,\y+5.7);
\node at (4.27,\y+4.27){$\bullet$};
\node[above] at (4.2,\y+4.4){$x_K$};
\node at (5.7,\y+5.7){$\bullet$};
\node[below] at  (5.9,\y+5.6){$x_L$};
\draw[line width = 0.8pt, <-] (6.2,\y+4) arc (165:90:1.3cm);
\node  at (9.5,\y+5){{\color{blue} $\sigma$} $ =K|L\in {\mathcal E}_\Omega^{\rm int}$};
\node  at (9.,\y+6){$K, L \in {\mathcal T}_\Omega^f$};
}
 \filldraw [fill=red!30!white] (3,-1) -- (6,2)--(8,-1)--cycle;
  \draw[line width = 0.8pt, <-] (7.2,-.9) arc (165:110:1.5cm);
 \node at (5.85,1.3){$K$};
 \node at (5.5,-1){$\bullet$};
 \node[below] at (5.5,-1){$x_{K^*}$};
 \node at (5.5,.3){$\bullet$};
 \node[right] at (5.5,.3){$x_K$};
 \draw[dashed, line width= 1.2pt] (5.5,-1)--(5.5,.3);
 \draw[line width = 1.2pt, color = red] (3,-1)--(8,-1);
\node[right] at (8.2,0.5){{\color{red} $K^*$} $ \in{\mathcal T}_\omega$};
\node[right] at (8.2, 1.2){$K\in {\mathcal T}_\Omega^r$};
\node[right] at (8.2, -0.3){{\color{red}{$\sigma$}} $= K|K^*\in {\mathcal E}_\Omega^r$};
\end{tikzpicture}
\end{tabular}
\caption {Presentation of the meshes and the associated notation in a 2D case for a rectangular field $\Omega$ and the unidimensional road $\omega$. The mesh of $\Omega$ is an admissible triangular mesh $\T_\Omega$. The control volumes in blue and white belong to $\T_{\Omega}^f $ while the control volumes in red belong to $\T_{\Omega}^r $. The control volumes of $\omega$, $K^*\in\T_\omega$, can also be considered as edges of $\E_\Omega^r$.}\end{figure}
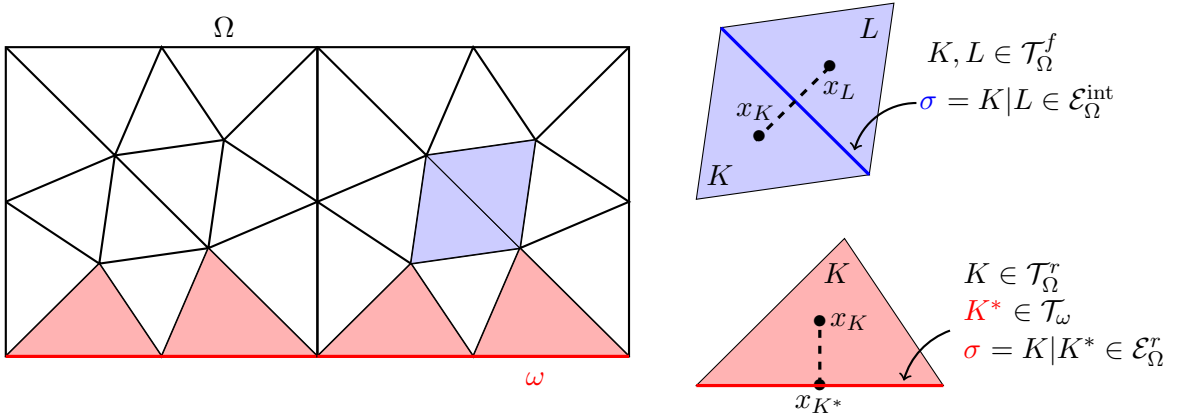

The measures of control volumes or edges are denoted by $m_{K}$, $m_{K^*}$, 
$m_\sigma$, $m_{\sigma^*}$ (which is set equal to 1 if the road has dimension 1). We also define by $d_{\sigma}$ or $d_{\sigma^*}$ the distance associated to an edge $\sigma\in \E_\Omega$ or $\sigma^*\in \E_{\omega}$, usually defined as the distance between the centers of two neighboring cells (or the distance from the center to the boundary), so that the transmissivities are defined by
$$
\tau_\sigma:=\ds\frac{m_\sigma}{d_\sigma}  \, \text{ for any } \sigma \in \E_\Omega,\quad 
\tau_{\sigma^*}:=\ds\frac{m_{\sigma^*}}{d_{\sigma^*}} \, \text{ for any } \sigma^* \in \E_\omega.
$$

Last, in view of time discretization, we consider a time step $\delta t>0$.

\medskip

\noindent{\bf The scheme.} Let us denote by $((v_K^n)_{K\in\T_{\Omega}, n\geq 0}, (\vKe^n)_{K^*\in\T_{\omega}, n\geq 1}, (\uKe^n)_{K^*\in\T_{\omega}, n\geq 0})$ the discrete unknowns. 
We start with the discretization of the initial conditions by letting
\begin{equation}\label{init_scheme}
v_K^0=\frac{1}{m_K}\int_K v_0(x,y) dx dy, \ \forall K\in\TO\ \mbox{ and }\  u_{K^*}^0=\frac{1}{m_{K^*}}\int_{K^*} u_0(x) dx,  \ \forall K^*\in \T_\omega.
\end{equation}


The scheme we propose is a backward Euler scheme in time and a TPFA finite volume scheme in space. It writes as 
\begin{subequations}\label{scheme}
\begin{align}
&	m_K\ds\frac{v_K^n-v_K^{n-1}}{\delta t} + d\ds\sum_{\sigma=K|L}\tau_\sigma (v_K^n-v_L^n)+ d\ds\sum_{\sigma=K|K^*}\tau_\sigma(v_K^n-\vKe^n)=0,\ \forall  K\in\T_{\Omega},\label{scheme.v}\\
&	-d\tau_\sigma(v_K^n-\vKe^n)=m_{K^*}(\mu \uKe^n-\nu \vKe^n),\ \forall \sigma\in \E_{\Omega}^{r}, \sigma =K|K^*,\label{scheme.interf}\\
&m_{K^*}\ds\frac{\uKe^n-\uKe^{n-1}}{\delta t}+D \ds\sum_{{\sigma^*}=K^*|L^*}\tau_{\sigma^*}(\uKe^n-u_{L^*}^n)
+m_{K^*}(\mu \uKe^n-\nu \vKe^n)=0,\ \forall K^*\in\T_{\omega}.\label{scheme.u}
\end{align}
\end{subequations}
At each time step, the scheme consists in a square linear system of equations of size $\# \T_{\Omega}+2\# \T_{\omega}$.

%

\subsection{Main results and organization of the paper}\label{ss:results}

In the present work we thus consider the field-road diffusion model in a bounded domain and study its convergence, at large times, to the steady-state selected by the initial data, in both the continuous \eqref{syst} and the discrete \eqref{scheme} setting. To do so, we  prove exponential decay of an entropy, see Theorem 
\ref{th:H2-decroit} and Theorem \ref{th:H2-decroit-dis}. Classically, this requires  to relate {\it dissipation} to the entropy via some functional inequalities.  However, the originality 
of this work comes from the difference of dimension between the field and the road and the exchange terms between both. In particular some refinements of Poincar\'e-Wirtinger inequality are required for the analysis, see Theorem \ref{th:unconv-PW} and Theorem \ref{th:unconv-PW-dis}.

\medskip

The paper is organized as follows. In Section \ref{s:continuous} we study the long time behavior of the  continuous model \eqref{syst}. In Section \ref{s:discrete} we study the long time behavior of the TPFA scheme \eqref{scheme}. Last, in Section \ref{s:num}, we perform some numerical simulations, that not only confirm the theoretical results but also raise new issues.

\section{Long time behavior of the continuous model}\label{s:continuous}

In this section, we consider  $v_0\in L^\infty(\Omega)$, $u_0\in L^\infty(\omega)$, both nonnegative and not simultaneously trivial. As a result, the total mass is initially positive
\begin{equation*}
\m:=\int _\Omega v_0(x,y)\,dxdy+\int_\omega u_0(x)\,dx >0.
\end{equation*}
Through this section, we denote $(v=v(t,x,y), u=u(t,x))$ the smooth  solution  of \eqref{syst} starting from $(v_0=v_0(x,y), u_0=u_0(x))$. 

\subsection{Mass conservation, positivity and steady-states}\label{ss:first prop}

First of all, it follows from the strong maximum principle, see \cite[Proposition 3.2]{Ber-Roq-Ros-13-1}, that both $v(t,x,y)$ and $u(t,x)$ are positive as soon as $t>0$.

Next, let us consider two test functions: $\varphi\in C^1(\R\times{\bar\Omega},\R)$ and $\psi\in C^1(\R\times\bar \omega,\R)$. We multiply the equation on $v$ in \eqref{syst} by $\varphi$ and the equation on $u$ by $\psi$ and we integrate respectively over $\Omega$ an $\omega$. After some integrations by parts, we obtain, due to the boundary conditions,
\begin{multline}\label{weak_form}
\int_\Omega \pa_t v\varphi (t,x,y) \,dx dy + \int_\omega \pa_t u \psi(t,x) \,dx=-d\int_\Omega \nabla v\cdot\nabla \varphi (t,x,y) \, dxdy \\
-D \int_\omega \nabla u\cdot \nabla \psi(t,x) \,dx -\int_\omega (\nu v(t,x,0)-\mu u(t,x))(\varphi(t,x,0)-\psi(t,x)) \,dx.\end{multline}
We emphasize that, in \eqref{weak_form} $\nabla v$ stands for $\nabla_{x,y} v$ while $\nabla u$ stands for $\nabla_x u$. 
In the sequel, we often omit the variables $t$, $x$ and $y$ in the integrands, when using \eqref{weak_form} or similar relations. When $v$ (or its derivatives) appears in an integrand over $\omega$, this obviously means $v|_{y=0}$.

Choosing constant functions equal to 1 over $\R\times{\bar\Omega}$ and $\R\times\bar \omega$ for $\varphi$ and $\psi$ in \eqref{weak_form}, we obtain that the total mass of the system  $\int_\Omega v(t,x,y)\,dxdy+\int_\omega u(t,x)\,dx$ is constant, namely
\begin{equation*}
\int_\Omega v(t,x,y)\,dxdy+\int_\omega u(t,x)\,dx=M_0, \quad \forall t>0.
\end{equation*}

Let us now investigate the existence of steady-states $(v=v(x,y), u= u(x))$ to \eqref{syst}. Using $\varphi=\nu v$ and $\psi=\mu u$ in \eqref{weak_form}, we obtain that 
$$
\int_\Omega |\nabla v|^2 \,dx dy= \int_\omega |\nabla u|^2 \,dx =\int_\omega (\nu v(\cdot,0)-\mu u(\cdot))^2 \,dx = 0,
$$
so that $v$ and $u$ must be constant in space and verify $\nu v-\mu u=0$. The system \eqref{syst} has thus an infinity of steady-states, but only one with the prescribed mass $\m$. The constant steady-state $(\vinf,\uinf)$ with mass $\m$ (and therefore associated to the initial state $(v_0,u_0)$) is given by
\begin{equation}\label{eq-pour-steady}
\nu \vinf-\mu \uinf=0, \quad \mO \vinf+\mo \uinf=\m,
\end{equation}
that is
\begin{equation}
\label{def-steady-cst}
\vinf=\frac{\mu}{\mo\nu + \mO\mu}\m, \quad \uinf=\frac{\nu}{\mo\nu +\mO\mu}\m.
\end{equation}
The positivity of $M_0$ implies the positivity of $\vinf$ and $\uinf$.

\subsection{Exponential decay of relative entropy}\label{ss:entropy}

Our aim is now to establish that  $(v=v(t,x,y), u=u(t,x))$, the smooth  solution  of \eqref{syst} starting from $(v_0=v_0(x,y), u_0=u_0(x))$ with an initial total mass $\m$, converges in large times towards the associated steady-state  $(\vinf,\uinf)$ defined by \eqref{def-steady-cst}. To do so, we apply a relative entropy method as presented for instance in the book by J\"ungel \cite{Jungel_2016}. 

For any twice differentiable function $\Phi:[0,+\infty)\to [0,+\infty)$ such that
$$
\Phi''>0, \quad \Phi'(1)=0, \quad \Phi(1)=0,
$$
we define an entropy, relative to the steady-state $(\vinf,\uinf)$, by
\begin{equation}
\label{def-H}
\H(t):=\int _\Omega \vinf \Phi\left(\frac{v(t,x,y)}{\vinf}\right)\,dxdy+\int_\omega \uinf \Phi \left(\frac{u(t,x)}{\uinf}\right)\,dx,
\end{equation}
which is obviously nonnegative and vanishes at time $t$ if and only if $v(t,\cdot,\cdot)\equiv \vinf$ and $u(t,\cdot)\equiv \uinf$.

Our first result states that the entropy is dissipated by the field-road model.

\begin{proposition}[Entropy dissipation]\label{prop:dissipation} 
Let $v_0\in L^\infty(\Omega)$ and $u_0\in L^\infty(\omega)$ be both nonnegative and satisfying $\m >0$. Let $(v=v(t,x,y), u=u(t,x))$ be the  solution  to \eqref{syst} starting from $(v_0=v_0(x,y), u_0=u_0(x))$, and $(\vinf,\uinf)$ the associated steady-state defined by \eqref{def-steady-cst}. Then the entropy defined by \eqref{def-H} is dissipated along time, namely
\begin{equation}\label{diss-entropy}
\frac{d}{dt}\H(t)=-\D(t)\leq 0, \quad \forall t>0,
\end{equation}
where 
\begin{multline}\label{def-Dissipation}
\D(t):=d \int_\Omega \frac{\vert \nabla v\vert ^2}{\vinf}\Phi''\left(\frac{v}{\vinf}\right)\,dxdy+D \int_\omega \frac{\vert \nabla u\vert ^2}{\uinf}\, \Phi''\left(\frac{u}{\uinf}\right) \,dx\\
\qquad +\mu \uinf\int _\omega \left(\Phi'\left(\frac{v}{\vinf}\right)-\Phi'\left(\frac{u}{\uinf}\right)\right)\left(\frac v{\vinf}-\frac u{\uinf}\right)\, dx
\end{multline}
is the so-called dissipation.
\end{proposition}

\begin{proof} 
The derivative of the entropy function $\H$ is given by 
$$
\frac{d}{dt} \H(t)= \int_\Omega \pa_t v \Phi'\left(\frac{v}{\vinf}\right) \, dxdy+ \int_\omega \pa_t u \Phi'\left(\frac{u}{\uinf}\right)\, dx.
$$
Therefore, we apply \eqref{weak_form} with $\varphi= \Phi'(\frac{v}{\vinf})$ and $\psi= \Phi'(\frac{u}{\uinf})$. 
Since $\nabla \varphi=\frac{1}{\vinf} \Phi''\left(\frac{v}{\vinf}\right)\nabla v$ and $\nabla \psi=\frac{1}{\uinf} \Phi''\left(\frac{u}{\uinf}\right)\nabla u$, we obtain 
\begin{multline*}
\frac{d}{dt} \H(t)=-d \int_\Omega \frac{\vert \nabla v\vert ^2}{\vinf}\, \Phi''\left(\frac{v}{\vinf}\right) \,dxdy
-D\int_\omega \frac{\vert \nabla u\vert ^2}{\uinf}\, \Phi''\left(\frac{u}{\uinf}\right) \,dx\\ -
\int_\omega (\nu v(t,x,0)-\mu u(t,x))\left(\Phi'\left(\frac{v(t,x,0)}{\vinf}\right)-\Phi'\left(\frac{u(t,x)}{\uinf}\right)\right)\,dx.
\end{multline*}
Using \eqref{eq-pour-steady}, this is recast
\begin{multline*}
\frac{d}{dt} \H(t)=-d \int_\Omega \frac{\vert \nabla v\vert ^2}{\vinf}\, \Phi''\left(\frac{v}{\vinf}\right) \,dxdy
-D\int_\omega \frac{\vert \nabla u\vert ^2}{\uinf}\, \Phi''\left(\frac{u}{\uinf}\right) \,dx\\ -
\mu u_{\infty} \int_{\omega}\left(\frac{v(t,x,0)}{\vinf}-\frac{u(t,x)}{\uinf}\right)\left(\Phi'\left(\frac{v(t,x,0)}{\vinf}\right)-\Phi'\left(\frac{u(t,x)}{\uinf}\right)\right)\,dx,
\end{multline*}
which concludes the proof.
%
\end{proof}

Proposition \ref{prop:dissipation} ensures that any relative entropy $\H$ is nonincreasing in time, while being nonnegative. In order to obtain the exponential decay in time of the entropy, we expect a  relation between the entropy and the dissipation of the form $\D\geq \Lambda \H$ for a given $\Lambda>0$. Indeed, combined with \eqref{diss-entropy}, this would ensure $\H(t)\leq \H(0)\exp(-\Lambda t)$.

We specify now the choice of the $\Phi$ function and therefore the entropy. We select 
\begin{equation}\label{phi-2}
\Phi(s)=\Phi_2(s):=\frac 12 (s-1)^2, 
\end{equation}
so that
\begin{equation}
\label{H-2}
\H_2(t)=\frac 12 \int_\Omega \frac{(v-\vinf)^2}{\vinf}\,dxdy+\frac 12 \int_\omega \frac{(u-\uinf)^2}{\uinf}\,dx,
\end{equation}
and
\begin{equation}\label{D-2}
\D_2(t)= d \int_\Omega \frac{\vert \nabla v\vert ^2}{\vinf } \,dxdy+ D \int_\omega \frac{\vert \nabla u\vert ^2}{\uinf } \,dx+\mu \uinf \int _\omega \left(\frac{v}{\vinf}-\frac{u}{\uinf}\right)^2\,dx.
\end{equation}

Theorem~\ref{th:unconv-PW} states the expected relation between the entropy $\H_2$ and its dissipation $\D_2$. Its proof will be given in the next subsection. As we will see, it is based on the proof of the Poincar\'e-Wirtinger inequality, while not being the combination of the classical Poincar\'e-Wirtinger inequality applied to $v$ and to $u$.

\begin{theorem}[Relating entropy and dissipation]
\label{th:unconv-PW} Let $v_0\in L^\infty(\Omega)$ and $u_0\in L^\infty(\omega)$ be both nonnegative and satisfying $M_0>0$. Let $(v=v(t,x,y), u=u(t,x))$ be the  solution  to \eqref{syst} starting from $(v_0=v_0(x,y), u_0=u_0(x))$,  and $(\vinf,\uinf)$ the associated steady-state defined by \eqref{def-steady-cst}. Then, for any $t>0$ (that we omit to write below),  there holds
\begin{eqnarray}
&& \frac 12 \int_\Omega \frac{(v-\vinf)^2}{\vinf}\,dxdy+\frac 12 \int_\omega \frac{(u-\uinf)^2}{\uinf}\,dx \nonumber \\
&&\qquad \leq \frac 1{\Lambda _2} \left(d \int_\Omega \frac{\vert \nabla v\vert ^2}{\vinf } \,dxdy+ D \int_\omega \frac{\vert \nabla u\vert ^2}{\uinf } \,dx+\mu \uinf \int _\omega \left(\frac{v}{\vinf}-\frac{u}{\uinf}\right)^2\,dx\right),\label{PW}
\end{eqnarray}
for some positive constant $\Lambda_2$ depending on the dimension $N$, the domain $\Omega$ (including $\omega$ and $L$),  the transfer rates $\mu$, $\nu$, and the diffusion coefficients $d$, $D$, see \eqref{Lambda-2} for further details.
\end{theorem}

As  \eqref{PW} means nothing else than $ \D_2(t)\geq {\Lambda _2} \H_2(t)$ for all $t>0$, we deduce from Theorem~\ref{th:unconv-PW} and Proposition~\ref{prop:dissipation}  the exponential decay  of the entropy $\H_2$, as stated in Theorem~\ref{th:H2-decroit}.
\begin{theorem}
[Exponential decay of entropy]\label{th:H2-decroit}  Let $v_0\in L^\infty(\Omega)$ and $u_0\in L^\infty(\omega)$ be both nonnegative and satisfying $\m>0$. Let $(v=v(t,x,y), u=u(t,x))$ be the  solution  to \eqref{syst} starting from $(v_0=v_0(x,y), u_0=u_0(x))$,  and $(\vinf,\uinf)$ the associated steady-state defined by \eqref{def-steady-cst}.  Then the entropy defined by \eqref{H-2} decays exponentially, namely
$$
0\leq \H_2(t)\leq \H_2(0) e^{-\Lambda _2 t}, \quad \forall t\geq 0,
$$
where $\Lambda_2$ comes from Theorem \ref{th:unconv-PW}. 
\end{theorem}

Due to the definition of $\H_2$, a direct consequence of Theorem  \ref{th:H2-decroit} is the exponential decay of $v$ ({\em resp.} u) towards $\vinf$ ({\em resp.} $ \uinf$) in $L^2(\Omega)$- ({\em resp.}  $L^2(\omega)$-) norm.

\subsection{Relating entropy and dissipation, proof of Theorem~\ref {th:unconv-PW}}

At first glance, the relation between entropy and dissipation \eqref{PW} has similarities with the Poincaré-Wirtinger inequality. However, if we 
 use the Poincaré-Wirtinger inequality twice, once for the term $\int_\Omega \vert \nabla v\vert ^2\,dxdy$ and once for the term $\int_\omega \vert \nabla u\vert ^2 \, dx$, then we fail to reconstruct $\H_2$ as a lower bound for the dissipation $\D_2$. We obviously have to take into account that the quantity that is preserved along time is the total mass $\m$; we also have to manage the fact that $v$ and $u$ are defined on domains with different dimension. Roughly speaking, we will first \lq\lq thicken the road'' from a subset of $\R^{N-1}$ to a subset of $\R^N$ and define an ``enlarged'' domain made of the field and the thickened road. On this enlarged domain, we may define a function based on $v$ on the field and $u$ on the thickened road.  Next, we follow the main steps of the Poincaré-Wirtinger classical inequality hoping that the constant does not blow up as the thickness tends to zero. It turns out that an additional term appears and that it is precisely the non-gradient term in $\D_2$. Finally, \eqref{PW} can be interpreted as a kind of unconventional Poincaré-Wirtinger inequality.
 
We start by recalling the very classical Poincaré-Wirtinger inequality in a bounded convex open set and take the liberty to present briefly the main steps of a possible proof.  
 
 To state this precisely, we define the dimensional constant
\begin{equation}
\label{C_N}
C_d:=\left\{\begin{array}{ll} \ln 2 & \text{ if } d=1,\\
\frac{2^{d-1}-1}{d-1}  & \text{ if } d\geq 2,\\
\end{array}
\right.
\end{equation}
which increases with $d$.

\begin{theorem}[Poincaré-Wirtinger inequality]\label{th:PW} Let $U$ be a bounded convex open set of $\R^d$ ($d\geq 1$). Let $f :U\to \R$ be a given function in $H^1(U)$. Define its mean as
$\moy:=\frac{1}{m_U}\int_ U f\,dx$. Then 
\begin{equation}
\label{ineg-PW}
\Vert f-\moy\Vert _{L^2(U)}^2=\frac 1 {2m_U}\iint  _{U^2} \left(f(x)-f(y)\right)^2\,dxdy \leq C_d\, (\text{Diam } U)^2\,  \int _U \vert \nabla f(z)\vert ^2\,dz,
\end{equation}
with $C_d$ the dimensional constant defined in \eqref{C_N}.
\end{theorem}

\begin{proof} The equality in \eqref{ineg-PW} is classical and can be straightforwardly checked. Next, for sufficiently smooth $f$ (the general case being later obtained by density arguments),
\begin{eqnarray*}
\iint  _{U^2} \left(f(x)-f(y)\right)^2\,dxdy
&=& \iint  _{U^2} \left(\int_0^1 \nabla f\left((1-t)x+ty\right)\cdot (x-y)\,dt\right)^2\,dxdy\\
 & \leq & \iint  _{U^2} \int_0^1 \vert \nabla f\left((1-t)x+ty\right)\cdot (x-y)\vert^2 \,dtdxdy \\
& \leq &   (\text{Diam } U)^2 \iint_{U^2} \int_0^1 \vert \nabla f\left((1-t)x+ty\right)\vert^2 \,dtdxdy .
\end{eqnarray*}
We cut the integral over $t\in(0,1)$ into two pieces and write
\begin{eqnarray*}
\iint_{U^2} \int_0^1 \vert \nabla f\left((1-t)x+ty\right)\vert^2 \,dtdxdy &\leq & \int_{y\in U} \int_0^{1/2} \int _{x\in U} \vert \nabla f\left((1-t)x+ty\right)\vert^2\,dxdtdy\\
&& + \int_{x\in U} \int_{1/2}^1 \int _{y\in U} \vert \nabla f\left((1-t)x+ty\right)\vert^2\,dydtdx
\\
&\leq & \int_{y\in U} \int_0^{1/2} \int _{z \in V_{t,y}} \vert \nabla f(z) \vert^2\,\frac{dz}{(1-t)^d}\,dtdy\\
&& + \int_{x\in U} \int_{1/2}^1 \int _{z\in W_{t,x} } \vert \nabla f (z)\vert^2\,\frac{dz}{t^d}\,dtdx.
\end{eqnarray*}
Since $U$ is convex both domains of integration over $z$, namely $V_{t,y}$ and $W_{t,x}$, are subset of $U$. Since $\int_{1/2}^1 t^{-d}\,dt$ is nothing else than the dimensional constant $C_d$ defined in \eqref{C_N}, we get
$$
\iint_{U^2} \int_0^1 \vert \nabla f\left((1-t)x+ty\right)\vert^2 \,dtdxdy\leq 2\,m_U\, C_d \int_U \vert \nabla f(z)\vert ^2\,dz.
$$
Putting all together, we get \eqref{ineg-PW}.
\end{proof}

Having in mind these classical moves, we now turn to the proof of the unconventional Poincaré-Wirtinger inequality.

\begin{proof}[Proof of Theorem \ref{th:unconv-PW}] For $\ell>0$, we \lq\lq enlarge'' $\Omega=\omega \times (0,L)$ to $\Omega^+=\omega\times (-\ell,L)$. We denote $\Omega_\ell =\omega \times (-\ell,0)$ the so-called thickened road. Reference points in $\Omega^+$ will be denoted $X=(x,y)$, $X'=(x',y')$, with $x$, $x'$ in $\omega$ and $y$, $y'$ in $(-\ell,L)$. We work with 
\begin{equation}
\label{d-sigma}
d\rho= \left(
\frac{\vinf}{\m}\mathbf 1_{\Omega}(x,y)+\frac{1}{\ell}\frac{\uinf}{\m}\mathbf 1_{\Omega_\ell}(x,y)
\right)\,dxdy,
\end{equation}
which is a probability measure as can  be checked thanks to \eqref{eq-pour-steady}, and with
\begin{equation}
\label{f(x,y)}
f(x,y)=\frac{v(x,y)}{\vinf}\mathbf 1_{\Omega}(x,y)+ \frac{u(x)}{\uinf}\mathbf 1_{\Omega_\ell}(x,y), \quad (x,y)\in \Omega^+=\omega \times (-\ell,L),
\end{equation}
where we have omitted the $t$ variable.

The point is that, as $\ell \to 0$,  the $L^\infty$ norm of the measure blows-up. Fortunately, as $\ell \to 0$, the domain $\Omega^+$ shrinks to $\Omega$ and moreover we only need to consider $f(x,y)$ given by \eqref{f(x,y)} (in particular $f$ is independent on $y$ in the thickned road). 

Observe first that
$$
 \moy:=\int_{\Omega ^+} f\, d\rho  = \frac 1 \m \int_\Omega v(x,y)\,dxdy+\frac{1}{\ell \m}\int_{\Omega _\ell } u(x)\,dxdy=\frac 1 \m \int_\Omega v\,dxdy+\frac{1}{\m}\int_{\omega } u\,dx= 1,
$$
and that
\begin{equation}
\label{1}
\Vert f -\moy\Vert ^2_{L^2(\Omega^+,d\rho)}=\int_\Omega \left(\frac{v}{\vinf}-1\right)^2\frac{\vinf}{\m}\,dxdy+\int_\omega \left(\frac{u}{\uinf}-1\right)^2\frac{\uinf}{\m}\,dx=\frac 2 \m \H_2(t).
\end{equation}

Next, similarly to the equality in \eqref{ineg-PW}, it is straightforward to check that 
\begin{equation}\label{2}
2 \Vert f-\moy\Vert ^2_{L^2(\Omega^+,d\rho)}= I_{\Omega,\Omega}+2I_{\Omega,\Omega_\ell}+I_{\Omega_\ell,\Omega_\ell},
\end{equation}
where
\begin{equation*}
I_{A,B}:=\int _{X\in A}\int_{X'\in B} \left(f(X)-f(X')\right)^2 \rho(X)\rho(X')\, dX'dX.
\end{equation*}

$(i)$ We start with the term
\begin{eqnarray*}
I_{\Omega,\Omega}&=&\iint_{(X,X')\in \Omega^2}\left(\frac{v(x,y)}{\vinf}-\frac{v(x',y')}{\vinf}\right)^2\frac{(\vinf)^2}{\m^2}\,dXdX'\\
&=& \frac{1}{\m^2} \iint _{(X,X')\in \Omega ^2}\left(v(x,y)-v(x',y')\right)^2 \,dXdX',
\end{eqnarray*}
and we are in the footsteps of the classical case. Using \eqref{ineg-PW} we get
\begin{equation}
\label{3}
I_{\Omega,\Omega}\leq \frac{1}{\m^2}2\mO C_{N}(\text{Diam } \Omega)^2 \int_\Omega \vert \nabla v\vert ^2\,dxdy=\frac{\mO \vinf }{\m^2}2 C_{N}(\text{Diam } \Omega)^2 \int_\Omega \frac{\vert \nabla v\vert ^2}{\vinf}\,dxdy.
\end{equation}

$(ii)$ Let us now turn to the term
\begin{eqnarray*}
I_{\Omega_\ell,\Omega_\ell}&=&\iint_{(X,X')\in \Omega_\ell^2}\left(\frac{u(x)}{\uinf}-\frac{u(x')}{\uinf}\right)^2\frac{1}{\ell ^2}\frac{(\uinf)^2}{\m^2}\,dXdX'\\
&=& \frac{1}{\m^2} \iint _{(x,x')\in \omega ^2}(u(x)-u(x'))^2\,dxdx',
\end{eqnarray*}
and, again, we are in the footsteps of the classical case. Using \eqref{ineg-PW} we get
\begin{equation}
\label{4}
I_{\Omega_\ell,\Omega_\ell}\leq \frac{1 }{\m^2}2\mo C_{N-1}(\text{Diam } \omega)^2 \int_\omega \vert \nabla u\vert ^2\,dx=\frac{\mo \uinf }{\m^2}2 C_{N-1}(\text{Diam } \omega)^2 \int_\omega \frac{\vert \nabla u\vert ^2}{\uinf}\,dx.
\end{equation}

$(iii)$ It remains to estimate the so-called unconventional term involving crossed terms, namely
\begin{eqnarray*}
I_{\Omega,\Omega_\ell}&=&\int_{X\in \Omega}\int_{X'\in \Omega_\ell}\left(\frac{v(x,y)}{\vinf}-\frac{u(x')}{\uinf}\right)^2 \frac{\vinf \uinf}{\m^2}\frac 1 \ell\, dX'dX\\
&=& \int_{X\in \Omega}\int_{x'\in \omega}\left(\frac{v(x,y)}{\vinf}-\frac{u(x')}{\uinf}\right)^2 \frac{\vinf \uinf}{\m^2}\, dx'dxdy.
\end{eqnarray*}
We can split it into three pieces, thanks to the following inequality: 
$$
\left(\frac{v(x,y)}{\vinf}-\frac{u(x')}{\uinf}\right)^2\leq 3\left(\left(\frac{v(x,y)}{\vinf}-\frac{v(x,0)}{\vinf}\right)^2 + \left(\frac{v(x,0)}{\vinf}-\frac{u(x)}{\uinf}\right)^2+ \left(\frac{u(x)}{\uinf}-\frac{u(x')}{\uinf}\right)^2\right).
$$
This implies $I_{\Omega,\Omega_\ell} \leq 3 (I_{\Omega,\Omega_\ell}^1+I_{\Omega,\Omega_\ell}^2+I_{\Omega,\Omega_\ell}^3)$ with obvious notations for these three terms, for which we now give a bound. For the first one, we have 
$$
\begin{aligned}
I_{\Omega,\Omega_\ell}^1&=\frac{\uinf}{\vinf\m^2}\int_{X\in \Omega}\int_{x'\in \omega}\left(v(x,y)-v(x,0)\right)^2 dx' dxdy,\\
&= \frac{\uinf\mo}{\vinf\m^2}\int_{x \in \omega}\int_{y \in (0,L)}\left(v(x,y)-v(x,0)\right)^2 \, dydx.
\end{aligned}
$$
We can then apply the classical procedure used to prove the one-dimensional Poincaré inequality, namely 
\begin{eqnarray*}
\int_{x \in \omega}\int_{y \in (0,L)}\left(v(x,y)-v(x,0)\right)^2 \, dydx&=& \int_{x\in \omega}\int_{y\in(0,L)} \left(\int _0 ^y \frac{\partial v}{\partial s}(x,s)\,ds\right)^2\,dydx\\
&\leq & \int_{x\in \omega}\int_{y\in(0,L)}  \int_0^y \left(\frac{\partial v}{\partial s}(x,s)\right)^2\,ds \times y \, dydx\\
&\leq & \frac{L^2}2 \int_{x\in \omega}\int_ 0 ^L \left(\frac{\partial v}{\partial s}(x,s)\right)^2\,dsdx,\\
\end{eqnarray*}
which yields 
$$
I_{\Omega,\Omega_\ell}^1\leq \frac{\uinf\mo L^2}{2\m^2}\int_\Omega \frac{\vert \nabla v\vert ^2}{\vinf}\,dxdy.
$$
Let us now consider 
$$
\begin{aligned}
I_{\Omega,\Omega_\ell}^2&=\frac{ \vinf \uinf}{\m^2} \int_{X\in \Omega}\int_{x'\in \omega}\left(\frac{v(x,0)}{\vinf}-\frac{u(x)}{\uinf}\right)^2\, dx'dxdy,\\
&=\frac{ \vinf \uinf\mO}{\m^2}\int_\omega \left(\frac{v(x,0)}{\vinf}-\frac{u(x)}{\uinf}\right)^2\,dx
\end{aligned}
$$
and we recover, up to a multiplicative constant, the non-gradient term in the definition of the dissipation, see \eqref{D-2}. Finally, for the third term, we have
$$
\begin{aligned}
I_{\Omega,\Omega_\ell}^3&=\frac{ \vinf \uinf}{\m^2}\int_{X\in \Omega}\int_{x'\in \omega}\left(\frac{u(x)}{\uinf}-\frac{u(x')}{\uinf}\right)^2\, dx'dxdy\,\\
&=\frac{ \vinf L}{\uinf \m^2}\int_{x\in \omega}\int_{x' \in \omega}\left(u(x)-u(x')\right)^2 \, dx'dx,
\end{aligned}
$$
which is nothing else that $\frac{\vinf}{ \uinf} L\times  I_{\Omega_\ell,\Omega_\ell}$, with $ I_{\Omega_\ell,\Omega_\ell}$ already estimated in \eqref{4}.
As a result, we obtain
\begin{multline}\label{5}
I_{\Omega,\Omega_\ell} \leq \frac 3 2 \frac{ \uinf}{\m^2} \mo L^2  \int_\Omega \frac{\vert \nabla v\vert ^2}{\vinf}\,dxdy
+3 \frac{ \vinf \uinf}{\m^2}\mO \int_\omega \left(\frac{v(x,0)}{\vinf}-\frac{u(x)}{\uinf}\right)^2\,dx\\
+6C_{N-1} \frac{\vinf}{\m^2}\mO (\text{Diam } \omega)^2 \int_\omega \frac{\vert \nabla u\vert^2}{\uinf}\, dx  .
\end{multline}

Now combining  \eqref{1}, \eqref{2}, \eqref{3}, \eqref{4} and \eqref{5} we reach 
\begin{eqnarray*}
\H_2(t)&\leq &\frac 1 4\frac{\mO}{\mo\nu+\mO\mu}\left(2\mu C_N (\text{Diam } \Omega)^2+ 3\nu L \right) \int _\Omega \frac{\vert \nabla v\vert^2}{\vinf}\,dxdy\\
&&+
\frac 1 2 \frac{\mo}{\mo\nu+\mO\mu}C_{N-1}  (\text{Diam } \omega)^2\left(\nu+6\mu L \right)\int_\omega \frac{\vert \nabla u\vert^2}{\uinf}\,dx\\
&&+\frac 3 2 \frac{\mO}{\mo\nu+\mO\mu} \mu \uinf \int_\omega \left(\frac{v}{\vinf}-\frac{u}{\uinf}\right)^2\,dx,
\end{eqnarray*}
where we have also used the relations \eqref{def-steady-cst} for $\vinf$ and $\uinf$.  Defining
\begin{multline}\label{Lambda-2}
\Lambda_2:=\min \Bigg\{\frac{4}{2\mu C_N (\text{Diam } \Omega)^2+ 3\nu L }\, \frac{\mo\nu+\mO\mu}{\mO} \, d\,;\\
 \frac{2}{C_{N-1}  (\text{Diam } \omega)^2\left(\nu+6\mu L\right)}\, \frac{\mo\nu+\mO\mu}{\mo}\, D
\,;\, \frac  2 3 \frac{\mo\nu+\mO\mu}{\mO}\Bigg\}, 
\end{multline}
we reach the conclusion of Theorem \ref{th:unconv-PW}. 
\end{proof}

\begin{remark}
Applying the following scaling in time 
$$
v(t,x,y)= V(dt,x,y),\ u(t,x)=U(dt,x),
$$
we obtain that $(V,U)$ is a solution to \eqref{syst} for the set of parameters $(d=1, D,\mu,\nu)$ if and only if $(v,u)$ is a solution to \eqref{syst} for the set of parameters $(d,dD,d\mu,d\nu)$, so that the {\it actual} decay rate should satisfy
$$
\Lambda(d,dD,d\mu,d\nu)=d\Lambda(1,D,\mu,\nu).
$$
With the scaling $v(t,x,y)=V(Dt,x,y), u(t,x)=U(Dt,x)$, we get 
$$
\Lambda(Dd,D,D\mu,D\nu)=D\Lambda(d,1,\mu,\nu). 
$$
We observe that $\Lambda_2$ given by \eqref{Lambda-2} satisfies these two expected scaling properties. 
\end{remark}

\section{Long time behavior of the TPFA scheme}\label{s:discrete}

In this section, we consider the finite volume scheme \eqref{init_scheme}-\eqref{scheme} for the field-road diffusion model. The assumptions on the initial data are the same as in the continuous case: $v_0\in\Linf(\Omega)$ and $u_0\in\Linf(\omega)$ are nonnegative and satisfy $M_0>0$. We also assume that the meshes ${\mathcal M}_\Omega$ and ${\mathcal M}_\omega$ satisfy the admissibility and compatibility assumptions introduced in subsection \ref{ss:TPFA}. We start with some preliminary results: existence and uniqueness of a solution to the scheme, positivity, mass conservation and steady-states. Then, we will focus on the long time behavior of the scheme. As in the continuous case, we will establish the exponential decay of the approximate solutions towards the steady-state, a result based on a discrete counterpart of the entropy-dissipation relation \eqref{PW}.

\subsection{Preliminary results}\label{ss:preliminary}

As already noticed in the introduction, the scheme \eqref{scheme} consists, at each time step, in a square linear system of equations of size $\# \T_{\Omega}+2\# \T_{\omega}$. 
We can obtain a weak formulation of the scheme by multiplying the equations in \eqref{scheme} by some test values and summing over $\T_\Omega$, $\E_{\Omega}^r$, $\T_{\omega}$. For a given vector
$$((\varphi_K)_{K\in\T_{\Omega}}, (\varphi_{K^*})_{K^*\in \T_{\omega}}, (\psi_{K^*})_{K^*\in \T_{\omega}}),
$$
we obtain
\begin{multline}\label{scheme_weak}
\ds\sum_{K\in\TO} m_K\phiK\frac{v_K^n-v_K^{n-1}}{\delta t}
+\ds\sum_{K^*\in\To} m_{K^*}\psiKe\frac{\uKe^n-\uKe^{n-1}}{\delta t}\\
+d\!\!\!\sum_{\sigma=K|K^*}\!\!\!\ts
(\vK^n-\vKe^n)(\phiK-\phiKe)=
-D\!\!\!\sum_{\sigma^*=K^*|L^*}\!\!\!\tse(\uKe^n-u_{L^*}^n)
(\psiKe-\psiLe)\\-d\!\!\!\sum_{\sigma=K|L}\!\!\!\ts (v_K^n-v_L^n)(\varphi_K-\varphi_L)
-\!\!\!\sum_{K^*\in\To}\!\!\! \mKe (\mu \uKe^n-\nu\vKe^n)(\psiKe-\phiKe).
\end{multline}
This weak formulation \eqref{scheme_weak} is equivalent to the scheme \eqref{scheme}. Indeed, setting one test value equal to 1 and the other ones equal to 0 permits to recover the scheme  \eqref{scheme} from \eqref{scheme_weak}.

\begin{lemma}[Well-posedness and basic facts] 
There exists a unique solution $$\left((v_K^n)_{K\in\T_{\Omega}, n\geq 0}, (\vKe^n)_{K^*\in\T_{\omega}, n\geq 1}, (\uKe^n)_{K^*\in\T_{\omega}, n\geq 0}\right)$$ to the scheme \eqref{init_scheme}-\eqref{scheme}. Moreover it  is nonnegative, positive as soon as $n\geq 1$,  and preserves the total mass  $\m$, namely
\begin{equation}\label{mass_cons}
\ds\sum_{K\in\TO} m_K \vK^n+\ds\sum_{K^*\in\To} m_{K^*}\uKe^n=\m, \quad \forall n\geq 0.
\end{equation}
\end{lemma}

\begin{proof} 

Assume $v_K^{n-1}=0$ for all $K\in\TO$ and $u_{K^*}^{n-1}=0$ for all $K^*\in\To$. Choosing $\phiK=\nu\vK^n$, $\phiKe=\nu \vKe^n$ and $\psiKe=\mu \uKe^n$ in \eqref{scheme_weak} yields existence and uniqueness of a solution to the scheme \eqref{scheme} at each time step. 

Assume now  $v_K^{n-1}\geq 0$ for all $K\in\TO$ and $u_{K^*}^{n-1}\geq 0$ for all $K^*\in\To$. Choosing $\phiK=\nu(\vK^n)^-$, $\phiKe=\nu (\vKe^n)^-$ and $\psiKe=\mu (\uKe^n)^-$  (where $x^-$ denotes the negative part of $x\in\R$) in \eqref{scheme_weak} yields by induction the nonnegativity of the solution to the scheme \eqref{scheme} at each time step:
$$
\vK^n\geq 0, \; \forall K\in\TO,\quad  \vKe^n\geq 0,\, \uKe^n\geq 0, \; \forall K^*\in\To.
$$

Now that we have established the nonnegativity of the solution to \eqref{scheme}, let us prove its positivity. Let $n\geq 1$ be given. Assume by contradiction that there is a $K_0\in \TO$ such that $v_{K_0}^n=0$ (the case $u_{K_0^*}^n=0$ for a $K_0^*\in \To$ being treated similarly). From  \eqref{scheme.v}, we deduce $v_L^n=0$ for all $L\in \TO$ neighbouring $K_0$ and $v_{K^*}^n=0$ for all $K^*\in \To$ bordering $K_0$. By repeating this we get $v_K^n=v_{K^*}^n=0$ for all $(K,K^*)\in \TO \times \To$. From  \eqref{scheme.interf}, we get $u_{K^*}^n=0$ for all $K^*\in \To$. By induction, we obtain $v_K^0=u_{K^*}^0=0$ for all $(K,K^*)\in \TO \times \To$, which is a contradiction. If there is a ${K_0^*}\in\To$ such that $v_{K_0^*}^n=0$, \eqref{scheme.interf} implies that $v_{K_0}^n=0$ for $K_0$ such that $K_0|K_0^* \in\E_{\Omega}^r$ and we come back to the preceding case. Finally, we obtain the positivity of the set of discrete solutions as soon as $n\geq 1$. 


Last, choosing the test vector constant equal to 1 in \eqref{scheme_weak} leads to the conservation of the total mass \eqref{mass_cons}.
\end{proof}

A steady-state is a solution of the form $\left((v_K^\infty)_{K\in\T_{\Omega}}, (\vKe^\infty)_{K^*\in\T_{\omega}}, (\uKe^\infty)_{K^*\in\T_{\omega}}\right)$, that is independent on $n$. Choosing $\phiK=\nu\vK^\infty$, $\phiKe=\nu \vKe^\infty$ and $\psiKe=\mu \uKe^\infty$ in \eqref{scheme_weak}, we get that the steady-state is actually constant in space: there are $\vinf \geq 0$ and $\uinf\geq 0$ such that  $\vK^\infty=\vinf=\vKe^\infty$ and $\uKe^\infty=\uinf$, for all $K\in \TO$, $K^*\in \To$. We also obtain that $\nu \vinf-\mu\uinf=0$ and, from the  mass conservation, that $m_\Omega \vinf+m_\omega \uinf=\m$. Finally, the steady-state of the scheme coincides with the steady-state of the continuous problem defined by \eqref{def-steady-cst}.

\subsection{Exponential decay of discrete relative entropy}\label{ss:entropy-dis}

As in the continuous case, we investigate the decay in time of some discrete relative entropies applied to the solution to the scheme. 
For any twice differentiable function $\Phi:[0,+\infty)\to [0,+\infty)$ such that
$\Phi''>0,\  \Phi(1)=0,\ \Phi'(1)=0$,
we define a discrete entropy, relative to the steady-state $(\vinf,\uinf)$, by 
\begin{equation}\label{entropy}
\HPhi^n:=\ds\sum_{K\in\TO} m_K \vinf \Phi (\ds \frac{\vK^n}{\vinf}) +\ds\sum_{\Ke\in\To} m_{K^*} \uinf \Phi(\ds\frac{\uKe^n}{\uinf}), \quad \forall n\geq 0.
\end{equation}
This is obviously the discrete counterpart of \eqref{def-H} and Proposition \ref{prop:dissipation-dis} states that it is dissipated along time.

\begin{proposition}[Entropy dissipation]\label{prop:dissipation-dis}
Let $$((v_K^n)_{K\in\T_{\Omega}, n\geq 0}, (\vKe^n)_{K^*\in\T_{\omega}, n\geq 1}, (\uKe^n)_{K^*\in\T_{\omega}, n\geq 0})$$ be the solution to the scheme \eqref{init_scheme}-\eqref{scheme}, and $(\vinf,\uinf)$ the associated steady-state defined by \eqref{def-steady-cst}. Then the discrete entropy defined by \eqref{entropy} is dissipated along time, namely
\begin{equation}\label{entropy-dissipation}
\ds\frac{\HPhi^n-\HPhi^{n-1}}{\delta t} \leq -\DPhi^n\ \leq 0, \quad \forall n\geq 1,\\
\end{equation}
where
\begin{multline}\label{dissipation}
\DPhi^n:= d\sum_{\sigma=K|K^*}\ts
(\vK^n-\vKe^n)\left(\Phi'(\frac{\vK^n}{\vinf})-\Phi'(\frac{\vKe^n}{\vinf})\right)
\\
+d\sum_{\sigma=K|L}\ts (v_K^n-v_L^n)\left(\Phi'(\frac{\vK^n}{\vinf})-\Phi'(\frac{v_L^n}{\vinf})\right)\\
+D\sum_{\sigma^*=K^*|L^*}\tse(\uKe^n-u_{L^*}^n)
\left(\Phi'(\frac{\uKe^n}{\uinf})-\Phi'(\frac{u_{L^*}^n}{\uinf})\right)\\
+\mu \uinf\sum_{K^*\in\To} \mKe  \left(\ds\frac{\uKe^n}{u^{\infty}}-\frac{\vKe^n}{\vinf}\right)\left(\Phi'(\frac{\uKe^n}{u^{\infty}})-\Phi'(\frac{\vKe^n}{\vinf})\right)
\end{multline}
is the so-called dissipation.
\end{proposition}

\begin{proof} Due to the convexity of $\Phi$, we have 
$$
\ds\frac{\HPhi^n-\HPhi^{n-1}}{\delta t}\leq \ds\sum_{K\in\TO} m_K\frac{v_K^n-v_K^{n-1}}{\delta t}\Phi'(\frac{\vK^n}{\vinf}) 
+\ds\sum_{K^*\in\To} m_{K^*}\frac{\uKe^n-\uKe^{n-1}}{\delta t} {\Phi'}(\frac{\uKe^n}{\uinf}).
$$
Then, we apply \eqref{scheme_weak} with 
$$\phiK=\Phi'(\ds\frac{\vK^n}{\vinf}),\  \phiKe=\Phi'(\ds\frac{\vKe^n}{\vinf}),\ \psiKe=\Phi'(\ds\frac{\uKe^n}{\uinf}),
$$ which leads to the entropy-dissipation relation \eqref{entropy-dissipation}, with the dissipation term $\DPhi^n$ rewritten as
\eqref{dissipation} thanks to \eqref{def-steady-cst}. Moreover the dissipation is nonnegative thanks to the monotonicity of $\Phi'$. 
\end{proof}

In the special case where $\Phi(s)=\Phi_2(s)=\frac 12 (s-1)^2$, we denote by $\Hd^n$ and $\Dd^n$ the corresponding entropy and dissipation at step $n$. They rewrite as 
\begin{eqnarray}
\Hd ^n&=&\ds\frac{1}{2} \ds\sum_{K\in\TO} m_K\frac{(\vK ^n-\vinf)^2}{\vinf}+\ds\frac{1}{2} \ds\sum_{K^*\in\To} m_{K^*}\frac{(\uKe ^n-\uinf)^2}{\uinf},\label{H-2-bis}\\
\Dd ^n&=&d\sum_{\sigma=K|K^*}\ts\frac{(\vK ^n-\vKe ^n)^2}{\vinf}+d\sum_{\sigma=K|L}\ts \frac{(v_K^n-v_L^n)^2}{\vinf}\nonumber \\
&&+D\sum_{\sigma^*=K^*|L^*}\tse\frac{(\uKe ^n-u_{L^*}^n)^2}{\uinf}+
\mu \uinf\sum_{K^*\in\To} \mKe  \left(\ds\frac{\uKe ^n}{u^{\infty}}-\frac{\vKe ^n}{\vinf}\right)^2.\label{D-2-bis}
\end{eqnarray}

We note that the relative entropy $\Hd$ corresponds to a weighted $L^2$ distance between the solution to the scheme and the constant steady-state having the same total mass, while the dissipation $\Dd$ corresponds to a weighted $L^2$ norm of a discrete gradient of the solution on the field and the road, with additional exchange terms.  

As in the continuous case, there exists a relation between the entropy $\Hd$ and its dissipation $\Dd$ given in Theorem~\ref{th:unconv-PW-dis}, which yields the exponential decay of $\Hd$ stated next in Theorem~\ref{th:H2-decroit-dis}. 


 \begin{theorem}[Relating entropy and dissipation]
\label{th:unconv-PW-dis} Let $$((v_K^n)_{K\in\T_{\Omega}, n\geq 0}, (\vKe^n)_{K^*\in\T_{\omega}, n\geq 1}, (\uKe^n)_{K^*\in\T_{\omega}, n\geq 0})$$ be the solution to the scheme \eqref{init_scheme}-\eqref{scheme},  and $(\vinf,\uinf)$ the associated steady-state defined by \eqref{def-steady-cst}. Then, there holds 
\begin{equation}\label{rel-HD}
\Hd ^n\leq \ds\frac{1}{\Lambda} \Dd ^n, \quad \forall n\geq 1,
\end{equation}
for some positive constant $\Lambda$ depending on the dimension $N$, the domain $\Omega$ (including $\omega$ and $L$),  the transfer rates $\mu$, $\nu$, and the diffusion coefficients $d$, $D$.
\end{theorem}


\begin{theorem}[Exponential decay of discrete entropy ]\label{th:H2-decroit-dis} Let $$((v_K^n)_{K\in\T_{\Omega}, n\geq 0}, (\vKe^n)_{K^*\in\T_{\omega}, n\geq 1}, (\uKe^n)_{K^*\in\T_{\omega}, n\geq 0})$$ be the solution to the scheme \eqref{init_scheme}-\eqref{scheme}, and $(\vinf,\uinf)$ the associated steady-state defined by \eqref{def-steady-cst}. Then the entropy defined by \eqref{H-2-bis} decays exponentially, namely
$$
0\leq \Hd^n\leq (1+\Lambda \,\delta t)^{-n} \Hd ^0, \quad \forall n\geq 0,
$$
where $\Lambda$ comes from Theorem \ref{th:unconv-PW-dis}.
\end{theorem}

Note that, as easily checked via Cauchy-Schwarz inequality, the initial discrete entropy $\Hd ^0$ is smaller than the initial continuous one, so that
$$
0\leq \Hd^n\leq (1+\Lambda \,\delta t)^{-n} \left( \frac{1}{2\vinf} \int_\Omega (v^0-\vinf)^2\, dxdy+\ds\frac{1}{2u^{\infty}}\int_\omega (u^0-u^{\infty})^2\,dx\right), \quad \forall n\geq 0.
$$
Moreover, the exponential decay of $\Hd$ implies the exponential decay of the discrete densities towards the steady-state in $L^2$.

\subsection{Relating discrete entropy and discrete  dissipation, proof of Theorem~\ref{th:unconv-PW-dis}}

As the proof of Theorem \ref{th:unconv-PW} is based on the proof of the Poincar\'e-Wirtinger inequality, the proof of Theorem \ref{th:unconv-PW-dis} is based on the proof of the discrete Poincar\'e-Wirtinger inequality given in \cite{Eym-Gal-Her-00}. The discrete Poincar\'e-Wirtinger inequality applies 
to functions which are piecewise constant in space on a bounded domain $U$ and therefore do not belong to $H^1(U)$. More precisely, if ${\mathcal M}=(\T,\E,\P)$ is an admissible mesh of $U$, we denote by $X(\T)$ the set of piecewise constant functions defined by 
$$
f\in X(\T) \Longleftrightarrow  \exists   (f_K)_{K\in \T} \in \R^\T,  \, f=\sum_{K\in\T} f_K {\mathbf 1}_K.
$$

 We start by recalling  in Lemma~\ref{lem:PW-dis} a key inequality in the proof of the discrete Poincaré-Wirtinger inequality, see (10.13) in  \cite[Proof of Lemma 10.2]{Eym-Gal-Her-00}. %
 

\begin{lemma}\label{lem:PW-dis}
Let $U$ be a polygonal bounded convex open set of $\R^d$ ($d\geq 1$) and ${\mathcal M}= (\T,\E,\P)$ be an admissible mesh of $U$. Then,   for any $f\in X(\mathcal T)$, we have
\begin{equation}\label{ineg-PW-dis}
\iint_{U^2} \left(f(x)-f(y)\right)^2 \,dxdy\leq  C_{d,U}\, (\text{Diam } U)^2\,   \sum_{\sigma=K|L} \tau_\sigma (f_K-f_L)^2,
\end{equation}
with $C_{d,U}$ the measure in $\R^d$ of balls of radius $\text{Diam } U$.
\end{lemma}

Having in mind this classical result, we now turn to the proof of the so-called unconventional discrete Poincaré-Wirtinger inequality \eqref{rel-HD}, which parallels that of Theorem \ref{th:unconv-PW}.

\begin{proof}[Proof of Theorem \ref{th:unconv-PW-dis}] For $\ell>0$, we \lq\lq enlarge'' $\Omega=\omega \times (0,L)$ to $\Omega^+=\omega\times (-\ell,L)$. We denote $\Omega_\ell =\omega \times (-\ell,0)$ the so-called thickened road. Reference points in $\Omega^+$ will be denoted $X=(x,y)$, $X'=(x',y')$, with $x$, $x'$ in $\omega$ and $y$, $y'$ in $(-\ell,L)$. Let us note that, based on the two meshes $\M_\Omega$ and $\M_\omega$, we can easily define a mesh of $\Omega^+$ just by ``enlarging'' the control volumes of $\omega$ to $\omega \times (-\ell,0)$. We work with the probability measure already defined in \eqref{d-sigma}, namely
$$
d\rho=\left(\ds\frac{\vinf}{\m} {\mathbf 1}_{\Omega}(x,y)+\frac{1}{\ell}\frac{\uinf}{\m} {\mathbf 1}_{\Omega_\ell}(x,y)\right)\,dxdy.
$$
We omit the time variable $n$ in the sequel. We consider $f$ the piecewise constant function on $\Omega^+$ defined by
\begin{equation}
\label{f(x,y)-dis}
f(x,y)=\ds\sum_{K\in\TO} \frac{\vK}{\vinf} {\mathbf 1}_K(x,y)+\ds\sum_{K^*\in\To} \frac{\uKe}{\uinf} {\mathbf 1}_{K^*\times (-\ell,0)}(x,y), \quad (x,y)\in \Omega^+.
\end{equation}
It satisfies 
$$
\moy:=\int_{\Omega ^+} f\, d\rho=1\mbox{  and  } \Hd =\frac{\m}{2}\Vert f-\moy\Vert_{L^2(\Omega^+, d\rho)}^2.
$$
Let us also introduce $v\in X(\T_\Omega)$, $u\in X(\T_{\omega})$ and $v^*\in X(\T_{\omega})$ defined by 
$$
v= \sum_{K\in\TO} \vK {\mathbf 1}_K,\quad u=\ds\sum_{K^*\in\To} {\uKe} {\mathbf 1}_{K^*},\quad v^*=\ds\sum_{K^*\in\To}{\vKe} {\mathbf 1}_{K^*}.
$$

As in the continuous case, see \eqref{1} and \eqref{2}, $\Hd$ can be rewritten as 
$$
\Hd= \frac{\m}{4}(I_{\Omega,\Omega}+2I_{\Omega,\Omega_\ell}+I_{\Omega_\ell,\Omega_\ell}),
$$
where
\begin{equation*}
I_{A,B}:=\int _{X\in A}\int_{X'\in B} \left(f(X)-f(X')\right)^2 \rho(X)\rho(X')\, dX dX'.
\end{equation*}

The terms $I_{\Omega, \Omega}$ and $I_{\Omega_\ell, \Omega_\ell}$ can be estimated  as in the proof of the discrete mean Poincar\'e inequality. 
Indeed, applying Lemma \ref{lem:PW-dis}, we get
\begin{eqnarray*}
I_{\Omega, \Omega}&=&\frac{1}{\m^2}\int _{X\in \Omega}\int_{X'\in \Omega}   \left(v(X)-v(X')\right)^2 \, dX dX'\\
&\leq &\frac{\vinf}{\m^2}C_{N,\Omega}  \, (\text{Diam } \Omega)^2\ds\sum_{\sigma=K|L}\ts \frac{(v_K-v_L)^2}{\vinf},
\end{eqnarray*}
and 
\begin{eqnarray*}
I_{\Omega_\ell, \Omega_\ell}&=&\frac{1}{\m^2}\int _{x\in \omega}\int_{x'\in \Omega} \left(u(x)-u(x')\right) \, dx dx'\\
&\leq & \frac{\uinf}{\m^2}C_{N-1,\omega}  \, (\text{Diam } \omega)^2\, \ds\sum_{\sigma^*=K^*|L^*}\tse\frac{(\uKe-u_{L^*})^2}{\uinf}.
\end{eqnarray*}

 It remains to estimate the so-called unconventional term involving crossed terms, namely $I_{\Omega, \Omega_\ell}$. We write
 \begin{eqnarray*}
I_{\Omega,\Omega_\ell}&=&\int_{X\in \Omega}\int_{X'\in \Omega_\ell}\left(f(X)-f(X')\right)^2 \frac{\vinf \uinf}{\m^2}\frac 1 \ell\, dX'dX\\
&=& \int_{X\in \Omega}\int_{x'\in \omega}\left(\frac{v(x,y)}{\vinf}-\frac{u(x')}{\uinf}\right)^2 \frac{\vinf \uinf}{\m^2}\, dx'dxdy.
\end{eqnarray*}
Introducing $v^*$ as in the continuous case, we obtain $I_{\Omega,\Omega_\ell}\leq 3(I_{\Omega,\Omega_\ell}^1+I_{\Omega,\Omega_\ell}^2+I_{\Omega,\Omega_\ell}^3)$, 
with 
\begin{eqnarray*}
I_{\Omega,\Omega_\ell}^1&=&\frac{\uinf\mo}{\vinf\m^2}\int_{x\in\omega}\int_{y\in (0,L)}(v(x,y)-v^*(x))^2 \, dy dx,\\
I_{\Omega,\Omega_\ell}^2&=&\frac{\vinf \uinf\mO}{\m^2} \int_{x\in\omega}\left(\frac{v^*(x)}{\vinf}-\frac{u(x)}{\uinf}\right)^2  \, dx,\\
I_{\Omega,\Omega_\ell}^3&=&\frac{\vinf L}{\uinf}I_{\Omega_\ell, \Omega_\ell}.
\end{eqnarray*}
Hence, the estimate of $I_{\Omega,\Omega_\ell}^3$ follows from that of $I_{\Omega_\ell, \Omega_\ell}$ above. Next, it is obvious that 
$$
I_{\Omega,\Omega_\ell}^2=\frac{\vinf \uinf\mO}{\m^2} \ds\sum_{\Ke\in\To} \mKe \left(\frac{\vKe}{\vinf}-\frac{\uKe}{\uinf}\right)^2,
$$
which is, up to a multiplicative constant, the non-gradient term  in the definition of the dissipation, see \eqref{D-2-bis}.

It thus only remains to estimate $I_{\Omega,\Omega_\ell}^1$. To do so, we adapt the proof of the discrete Poincaré inequality in \cite[Lemma 10.2]{Eym-Gal-Her-00}. For $\sigma \in \E_\Omega$, we define the function $\chi_{\sigma}:\Omega=\omega\times (0,L) \to \{0, 1\}$ by 
$$
\chi_\sigma(x,y):=\left\{\begin{array}{ll} 1 & \text{ if $\sigma$ intersects the vertical segment connecting $(x,y)$ to $(x,0)$},\\
0  & \text{ if not}.\\
\end{array}
\right.
$$
Therefore, for all $x\in\omega$, for all $y\in (0,L)$, 
$$
\vert v(x,y)-v^*(x)\vert \leq \sum_{\sigma =K|L} \vert v_K-v_L\vert \chi_{\sigma}(x,y)+\sum_{\sigma =K|K^*}\vert v_K-\vKe\vert  \chi_{\sigma}(x,y).
$$
Let $c_{\sigma}=\vert {\mathbf e}\cdot {\mathbf n}_{\sigma}\vert $ where $e$ is a unit vector of the vertical line and ${\mathbf n}_{\sigma}$ is a unit normal vector to $\sigma$. From Cauchy-Schwarz inequality, we have 
\begin{multline}\label{diff-vv*}
( v(x,y)-v^*(x))^2\leq \left(\sum_{\sigma =K|L} \frac{(v_K-v_L)^2}{d_{\sigma} c_{\sigma}}\chi_{\sigma}(x,y)+\sum_{\sigma =K|K^*}\frac{( v_K-\vKe)^2}{d_{\sigma} c_{\sigma}} \chi_{\sigma}(x,y)\right)\\
\times \left(\sum_{\sigma =K|L}d_{\sigma} c_{\sigma}\chi_{\sigma}(x,y)+\sum_{\sigma =K|K^*}d_{\sigma} c_{\sigma}\chi_{\sigma}(x,y)\right).
\end{multline}
Similarly as in the proof of  \cite[Lemma 10.2]{Eym-Gal-Her-00}, the second factor in the above right-hand-side is smaller than the \lq\lq vertical diameter'', that is $L$. Let us now integrate \eqref{diff-vv*} over $(x,y)\in \omega\times (0,L)$. Noticing  that
$$
\int_{x\in\omega}\int_{y\in (0,L)}\chi_{\sigma}(x,y)\, dxdy \leq m_{\sigma} c_{\sigma} L,
$$
we get 
$$
I_{\Omega,\Omega_\ell}^1\leq\frac{\uinf \mo}{\vinf \m^2} L^2 \left(\sum_{\sigma =K|L}\tau_\sigma (v_K-v_L)^2 + \sum_{\sigma =K|K^*} \tau_\sigma (v_K-\vKe)^2\right).
$$

Gathering all the bounds, we obtain \eqref{rel-HD} with a  decay rate
$$
\Lambda=\min(c_1d; c_2D; c_3),
$$
where the $c_i$'s ($1\leq i\leq 3$) are positive constants depending only on $N$, $\Omega$, $\mu$ and $\nu$.
\end{proof}

\section{Numerical experiments}\label{s:num}

For all the numerical experiments performed in this section, we consider the one-dimensional road
$\omega=(-2L,2L)$, and the two-dimensional field $\Omega=\omega\times(0,L)$, where we fix $L=20$.

\subsection{Test cases and profiles}\label{ss:test}

In this subsection we choose one set of parameters, but consider different initial conditions that lead to different test cases. Recently, the role of the founding population, in particular fragmentation, on the success rate of an invasion has received a lot of attention, see \cite{Dru-et-al-07}, \cite{Gar-Roq-Ham-12}, \cite{Maz-Nad-Tol-21}, \cite{Alf-Ham-Roq-preprint} and the references therein. Related to this question, we want to check how convergence to the steady state depends on the initial distribution of individuals, using the following four test cases.

\medskip

\noindent {\bf Test case 1.} Initially the road is empty and individuals are grouped together in the field, say: 
$$
\left\{\begin{aligned}
v_0(x,y)&=100\cdot{\mathbf 1}_{[-2.5,2.5]\times [2.5,7.5]}(x,y),\\
u_0(x)&=0.
\end{aligned}
\right.
$$

\noindent {\bf Test case 2.} Initially, individuals are grouped together on the road and in the field, say:
$$
\left\{\begin{aligned}
v_0(x,y)&=150\cdot {\mathbf 1}_{[-2.5,2.5]\times [2.5,5]}(x,y),\\
u_0(x)&=125\cdot {\mathbf 1}_{[-2.5,2.5]}(x).
\end{aligned}
\right.
$$

\noindent {\bf Test case 3.} Initially, the road is  empty and individuals are scattered in the field, say:
$$
\left\{\begin{aligned}
&v_0(x,y)=100\cdot {\mathbf 1}_{[-10,-7.5]\cup[-5,-2.5]\cup[2.5,5]\cup[7.5,10]}(x)\cdot {\mathbf 1}_{[7.5,10]}(y),\\
&u_0(x)=0.
\end{aligned}
\right.
$$

\noindent {\bf Test case 4.} Initially, individuals are scattered on the road and in the field, say:
$$
\left\{\begin{aligned}
&v_0(x,y)=150\cdot{\mathbf 1}_{[-10,-7.5]\cup[-5,-2.5]\cup[2.5,5]\cup[7.5,10]}(x)\cdot {\mathbf 1}_{[8.75,10]}(y),\\
&u_0(x)=62.5\cdot {\mathbf 1}_{[-10,-7.5]\cup[-5,-2.5]\cup[2.5,5]\cup[7.5,10]}(x).
\end{aligned}
\right.
$$

Note that all these initial conditions lead to the same total mass $\m=2500$. We fix $\mu=1$ and $\nu=5$ so that they have the same steady state, namely $(\vinf,\uinf)=(1.25,6.25)$. We also set the diffusion coefficient in the field to $d=1$ and on the road to $D=1$. 
The mesh we use for the simulations is a triangular mesh of 14336 triangles and we choose a time step equal to $10^{-1}$.

Figures~\ref{fig:CT12_champ} and \ref{fig:CT12_route} show the density profiles in the field and on the road respectively for Test Cases 1 and 2 at different times, while Figures~\ref{fig:CT34_champ} and \ref{fig:CT34_route} show the density profiles for Test Cases 3 and 4 at different times. 

We observe that the solutions to Test Cases 1 and 2 quickly show comparable behavior, as do the solutions to Test Cases 3 and 4.  This suggests that the initial presence or absence of individuals on the road has little effect on the solutions. 

On the other hand, we observe that the homogenization is slightly faster in Test Cases 3-4 than in Test Cases 1-2 (compare at $T=50$ for $v$ and at $T=100$ for $u$). The reason is that Test Cases 3-4 correspond to more fragmented founding populations for which it is easier to invade the whole domain (especially in the presence of moderate diffusion coefficients as chosen here, $D=d=1$).

The first observations above deal with  transient dynamics. In order to compute the decay rates towards the steady state, we have to move to larger time horizons. 
Figure \ref{fig:CT1to4_entropies} shows the time decay of the relative entropies (divided by the value at the first time step) for the four test cases. The computations were stopped when the computed ratio of the relative entropies reached the value $10^{-5}$.

We observe that the computed decay rate, according to Theorems \ref{th:H2-decroit} and \ref{th:H2-decroit-dis}, is the same for the four test cases, namely $\Lambda_{num}= 0.0123$, even if the transient behavior is slightly different, as already observed above.

\begin{figure}[!htb]
\begin{center}
\begin{tabular}{cc}
 \includegraphics[width=0.45\textwidth]{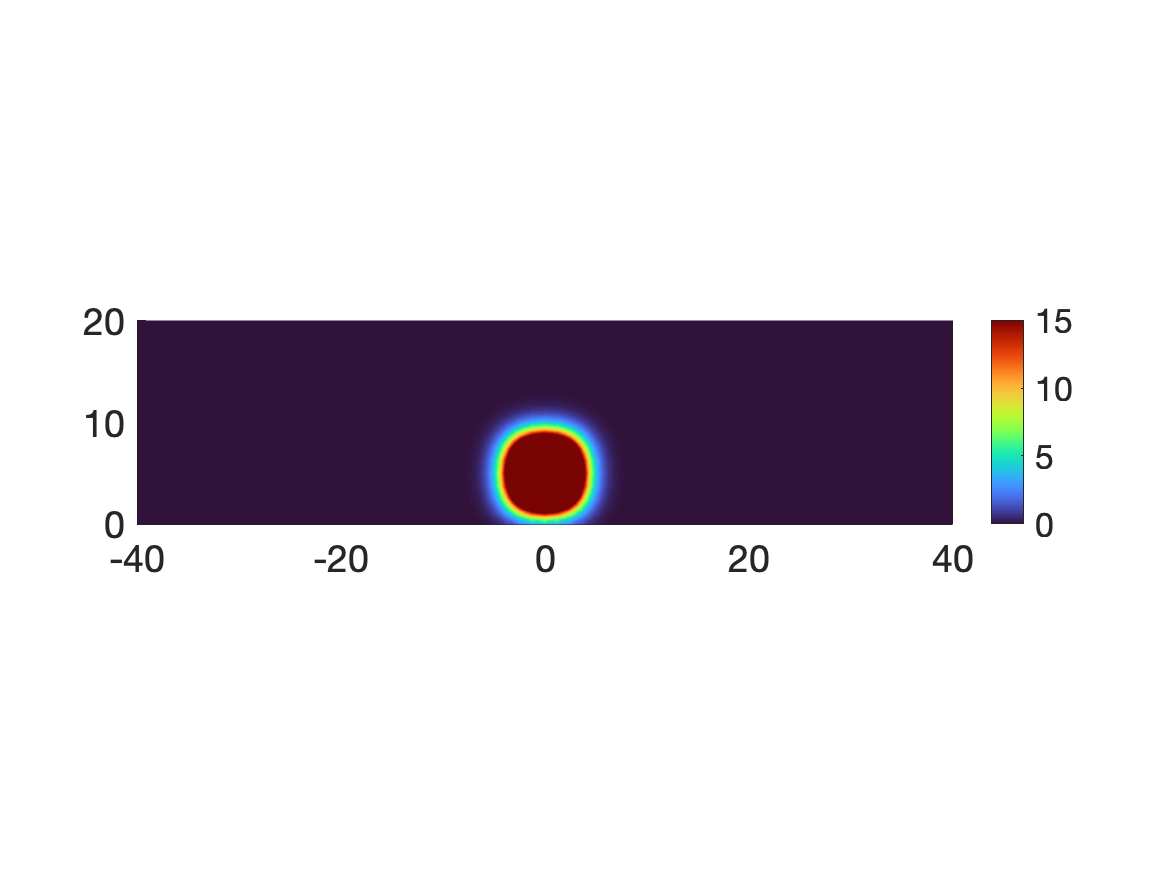}
 &\includegraphics[width=0.45\textwidth]{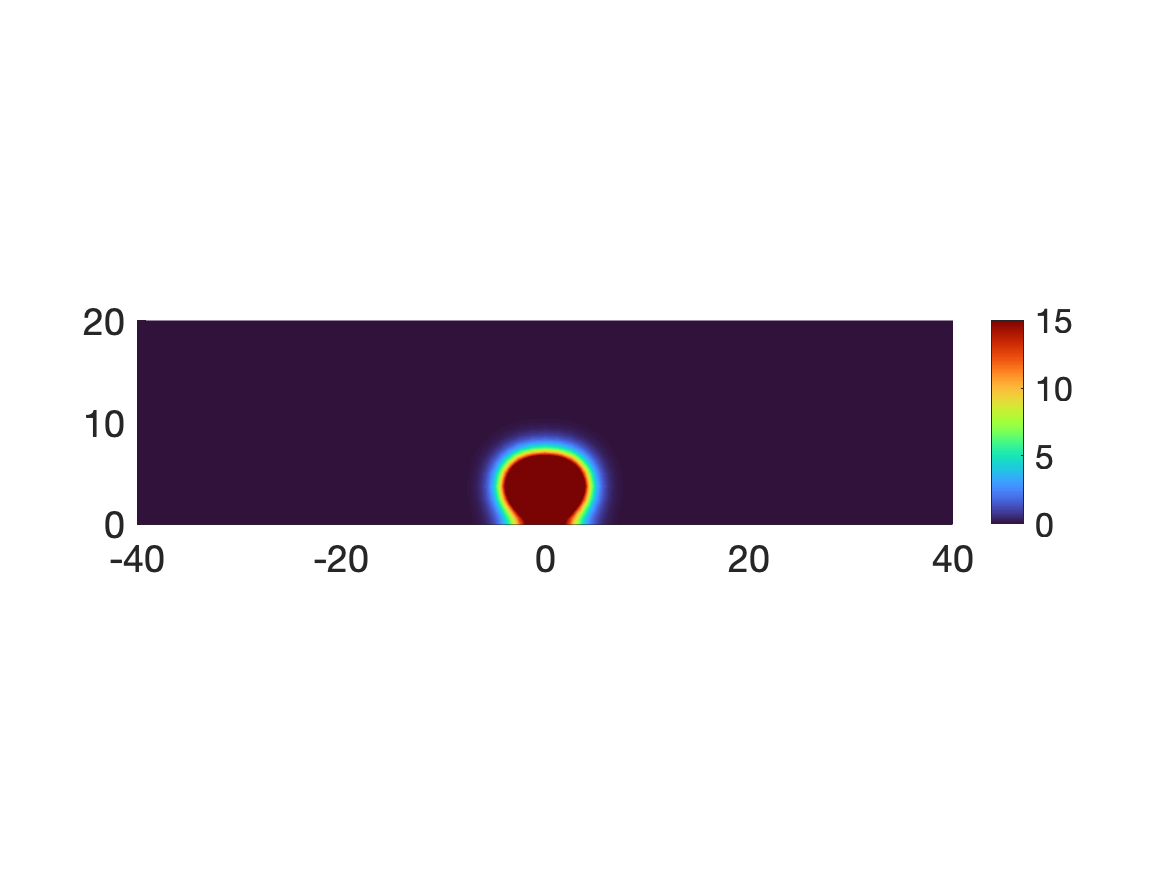}\\[-1.8cm]
 \includegraphics[width=0.45\textwidth]{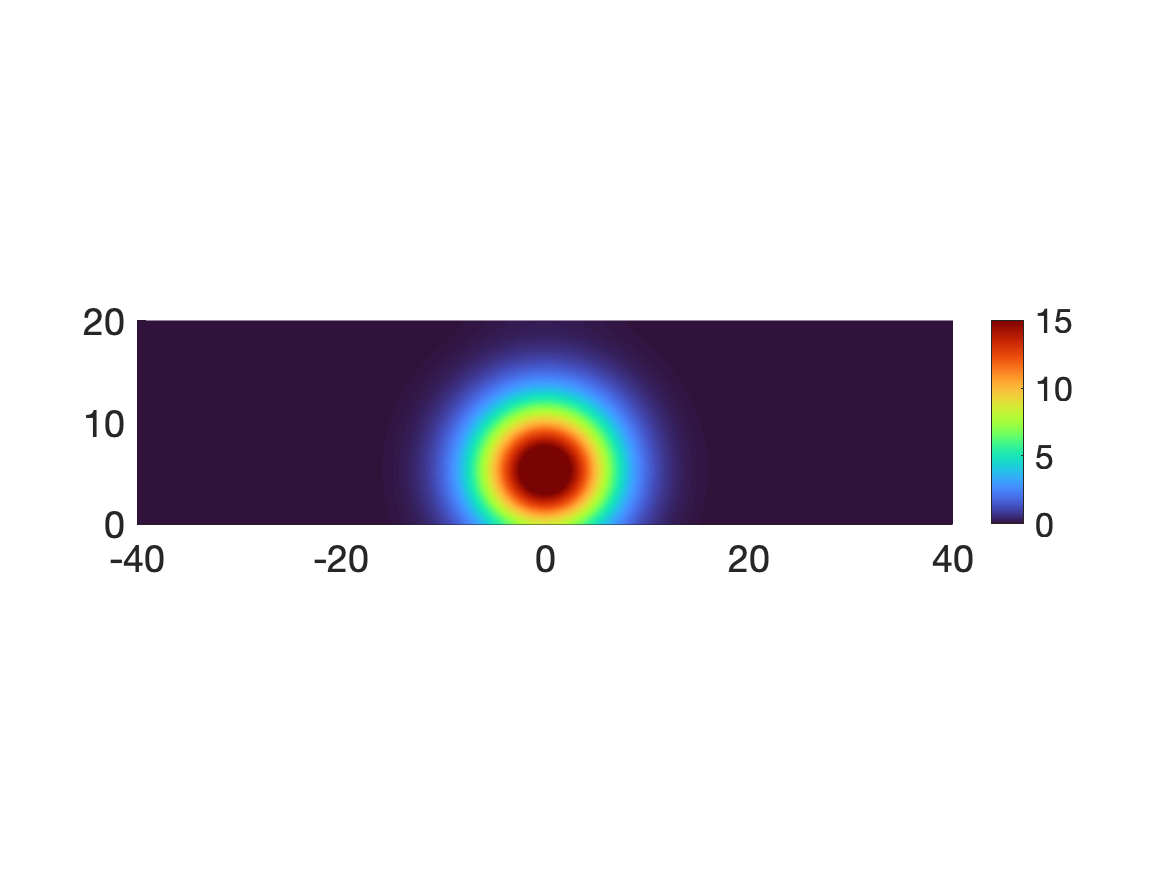}
 &\includegraphics[width=0.45\textwidth]{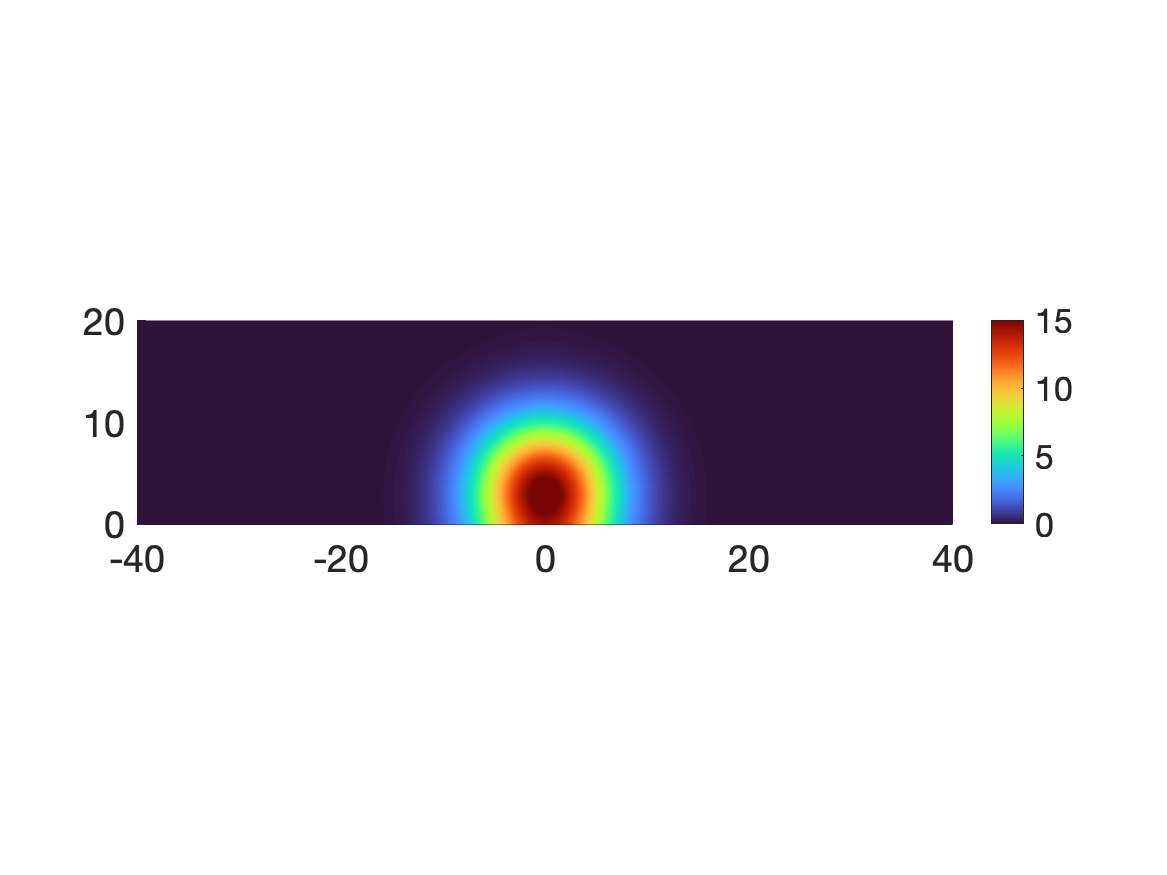}\\[-1.8cm]
\includegraphics[width=0.45\textwidth]{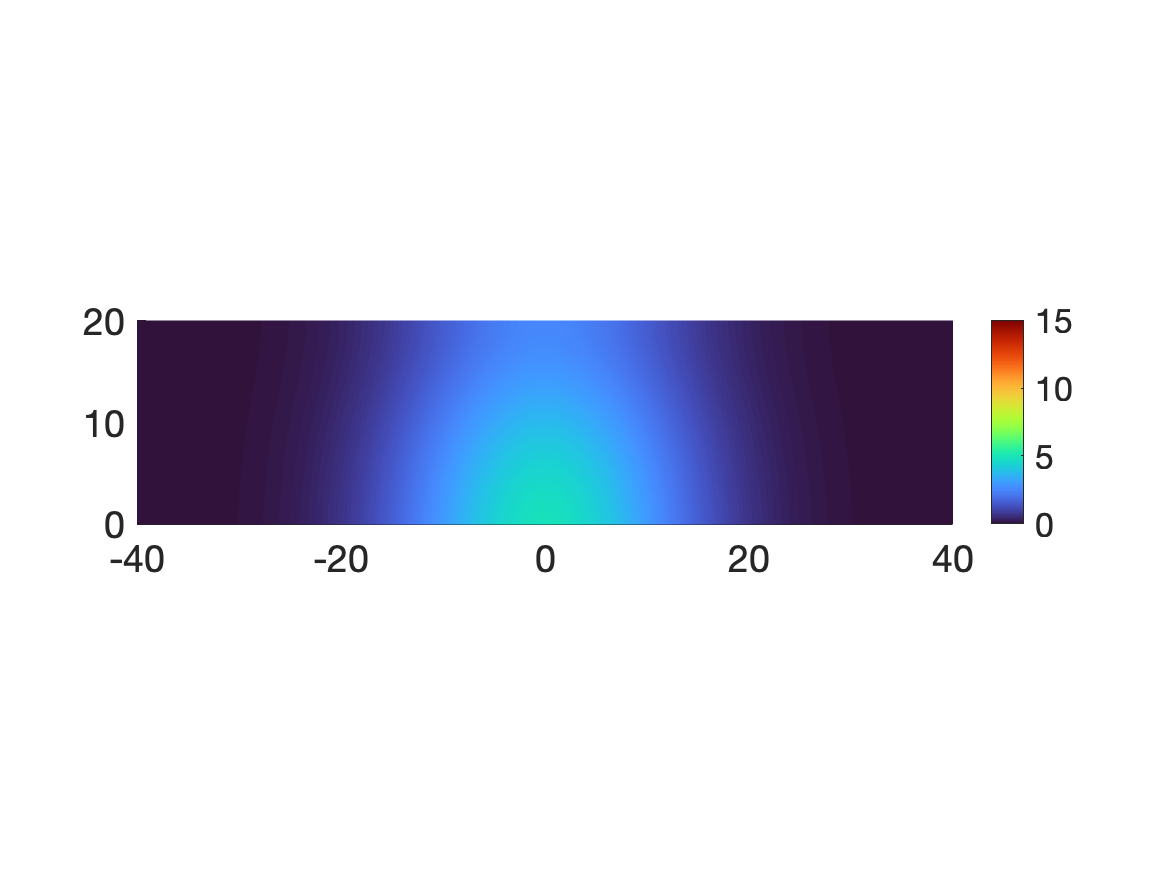}
 &\includegraphics[width=0.45\textwidth]{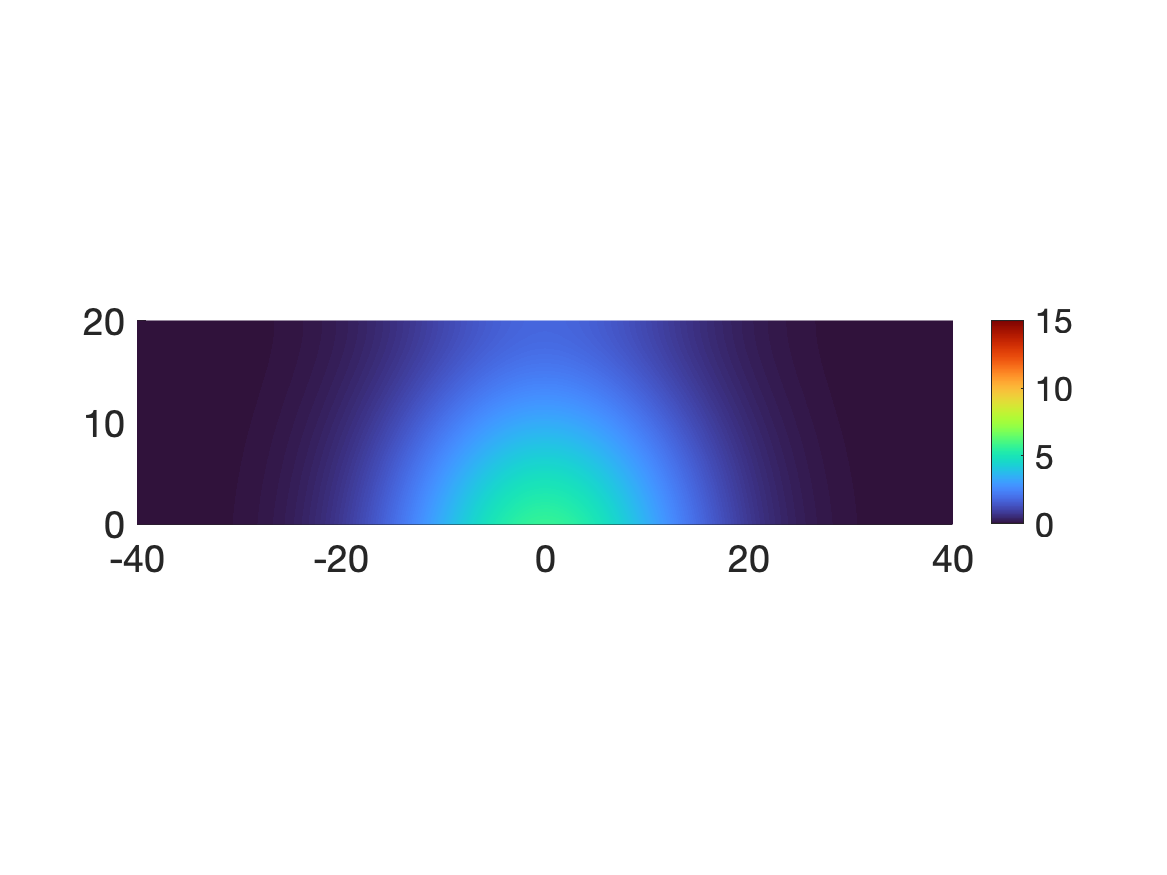}\\[-1.8cm]
\includegraphics[width=0.45\textwidth]{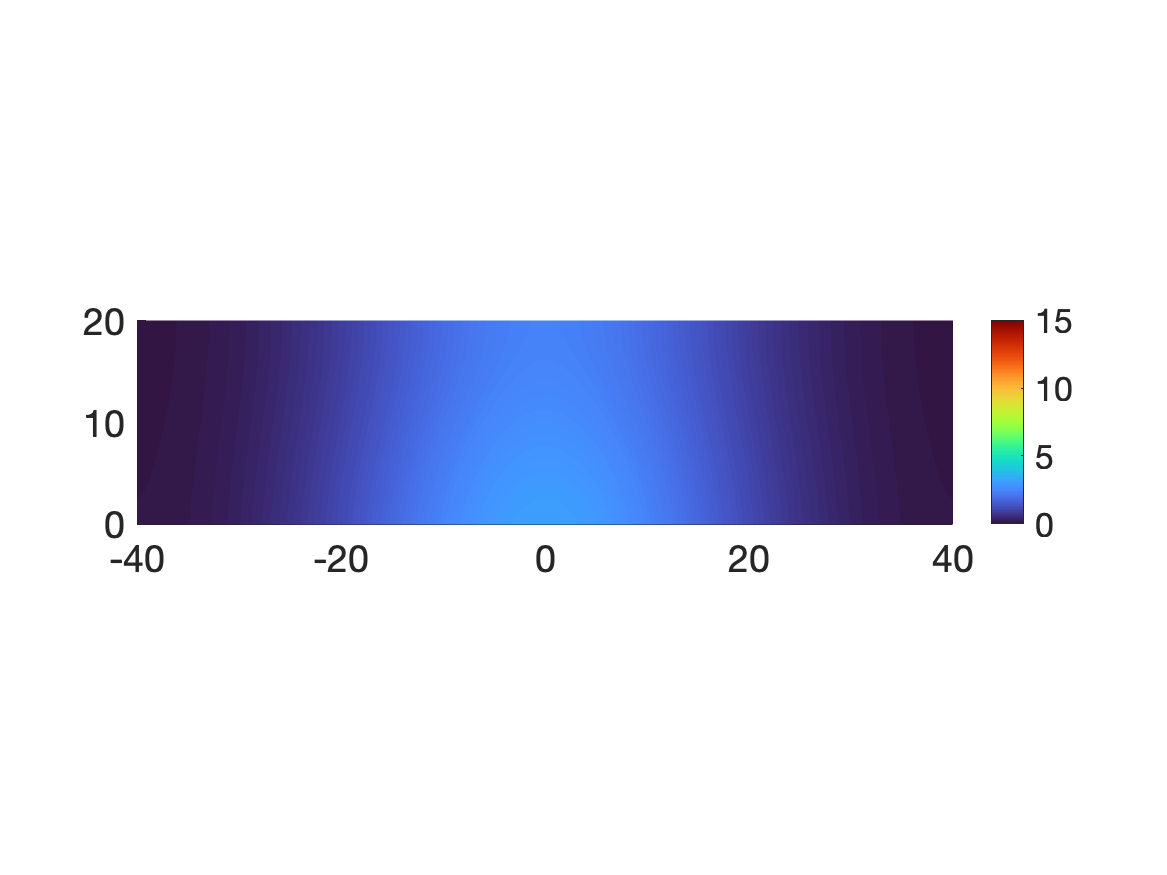}
 &\includegraphics[width=0.45\textwidth]{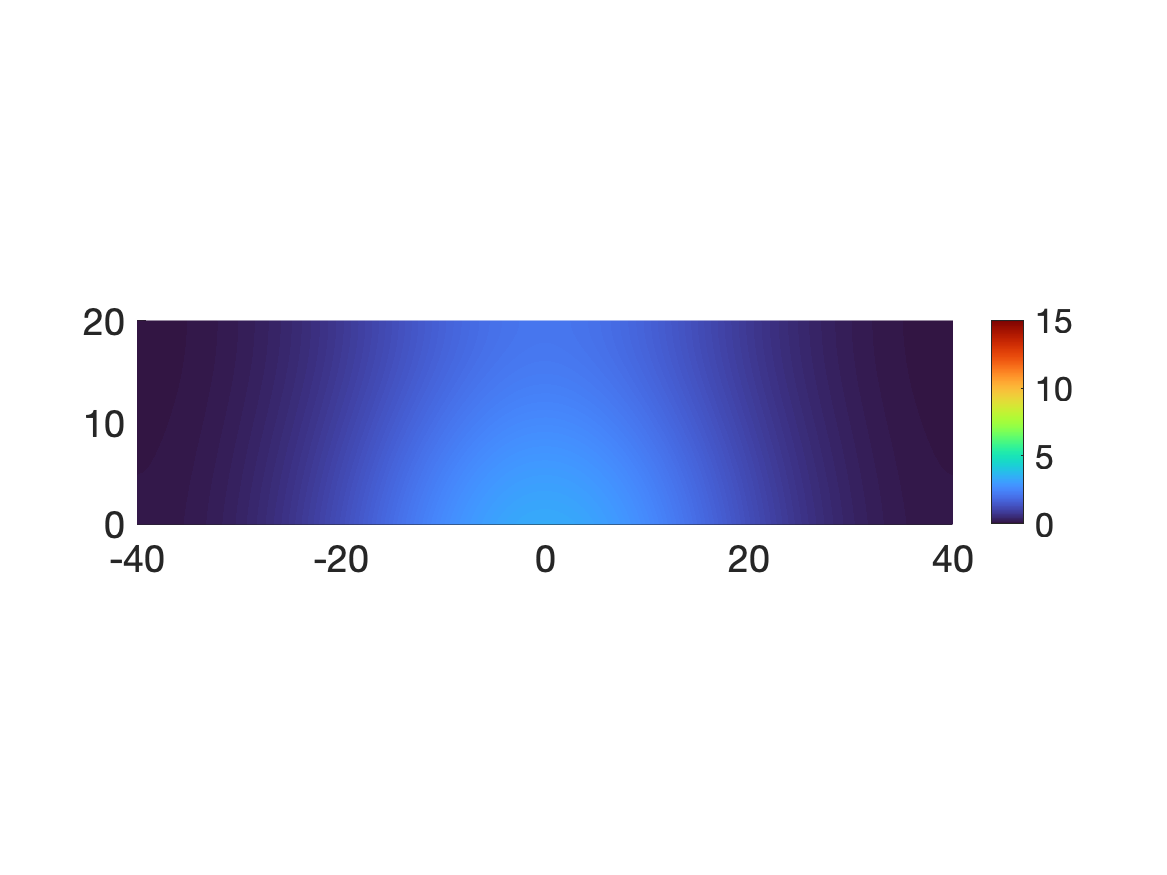}\\[-2.cm]
\end{tabular}\end{center}
\caption{Profiles of $v$, density in the field, for Test Case 1 (on the left) and Test Case 2 (on the right) at different times : $T=1$, $T=10$, $T=50$, $T=100$ (from the top to the bottom).}
\label{fig:CT12_champ}
\end{figure}

\begin{figure}[!htb]
\begin{center}
\begin{tabular}{cc}
 \includegraphics[width=0.35\textwidth]{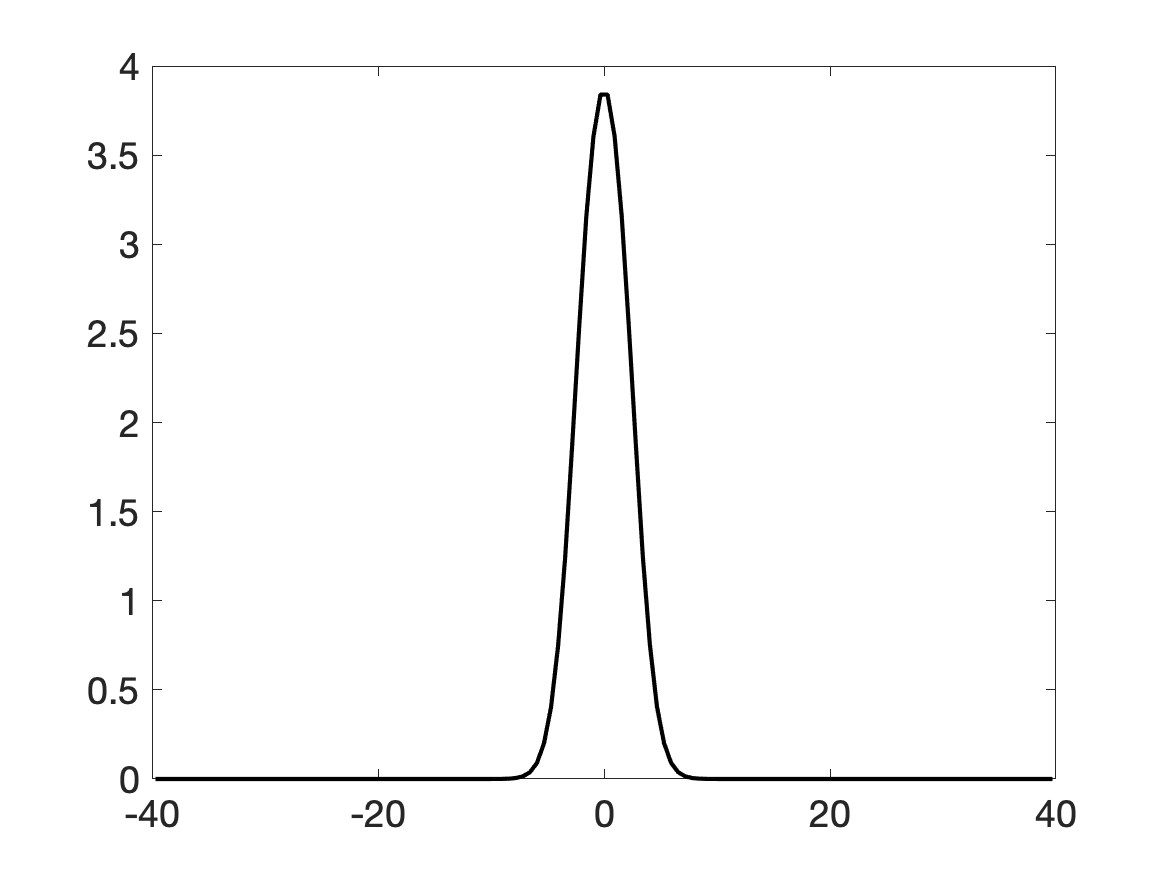}
 &\includegraphics[width=0.35\textwidth]{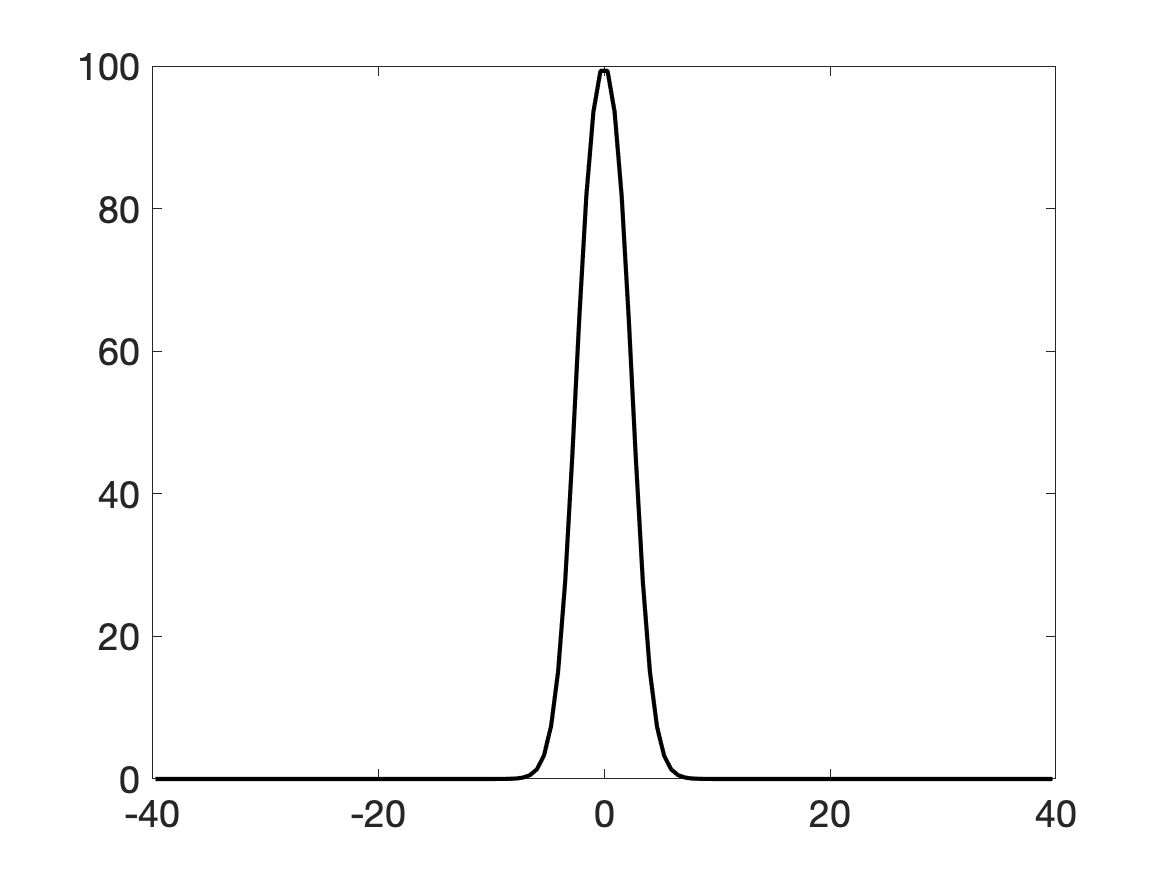}\\
 \includegraphics[width=0.35\textwidth]{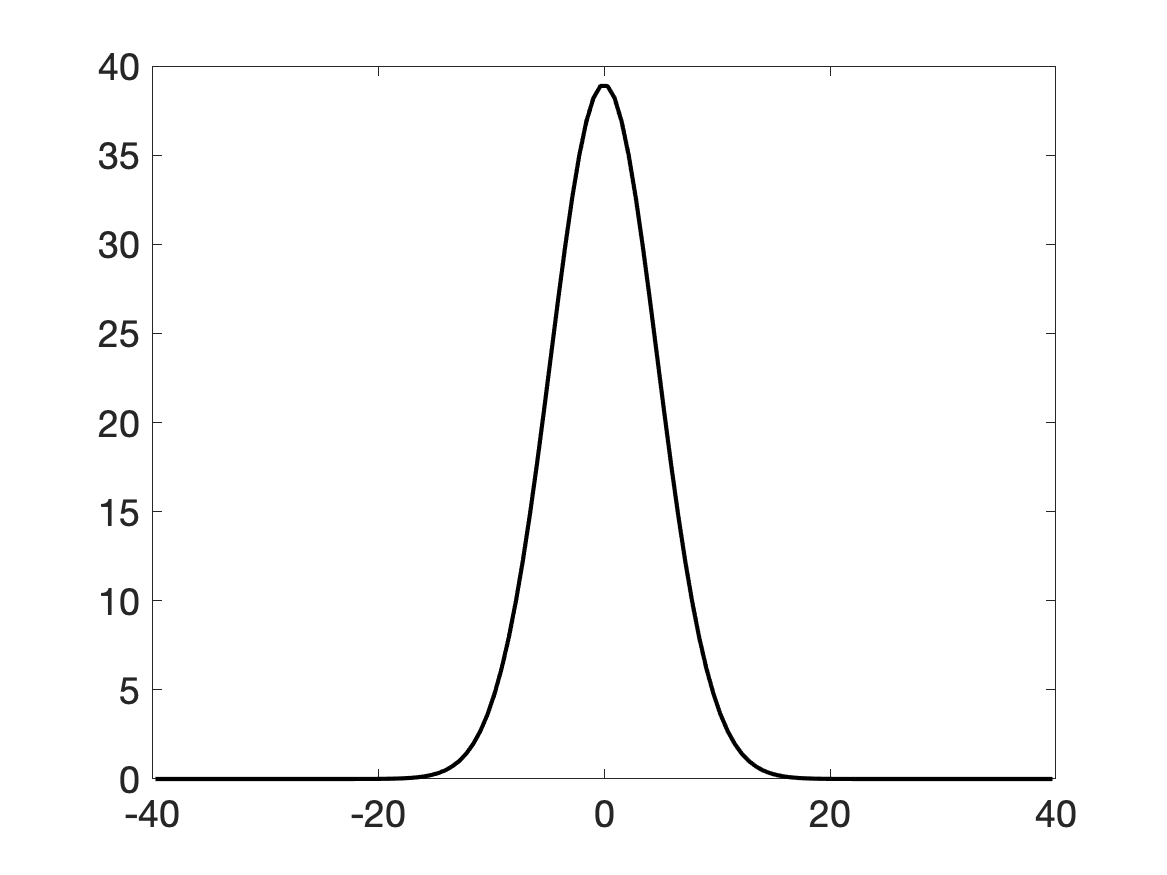}
 &\includegraphics[width=0.35\textwidth]{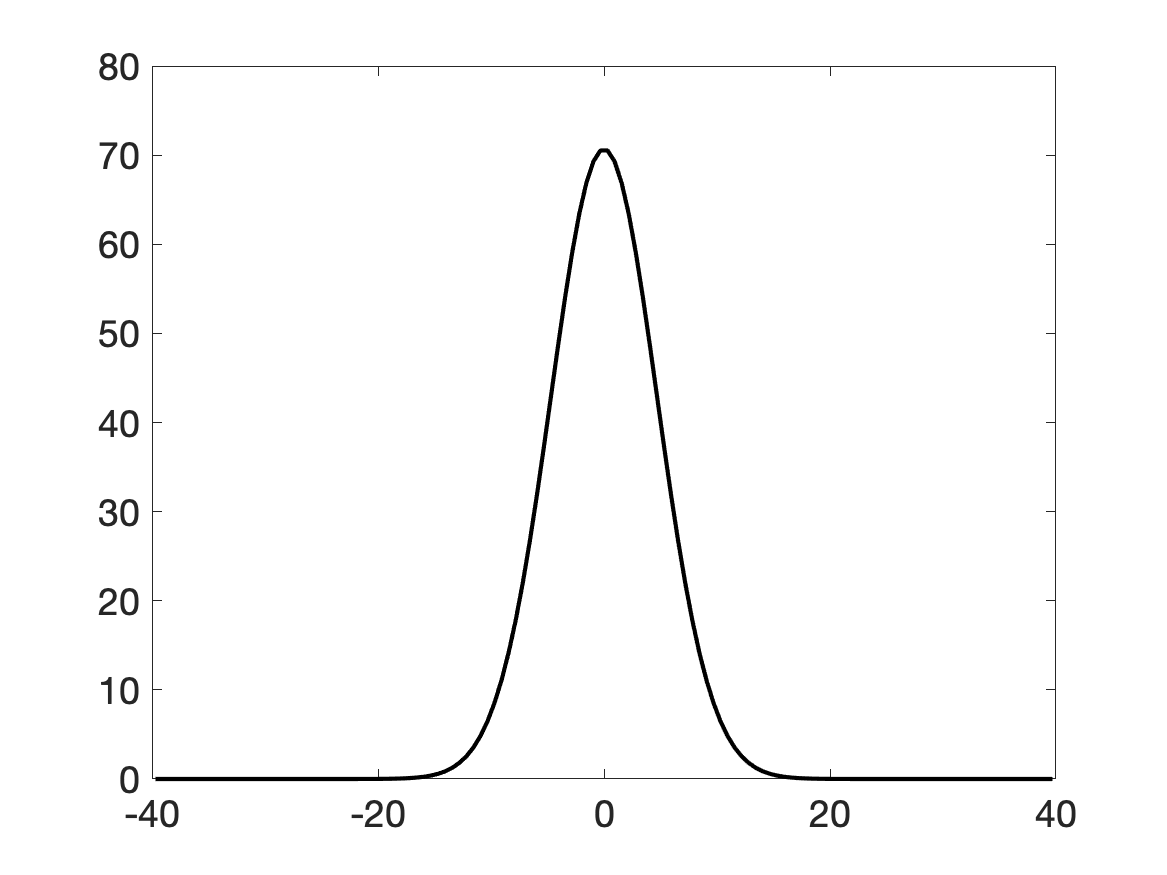}\\
\includegraphics[width=0.35\textwidth]{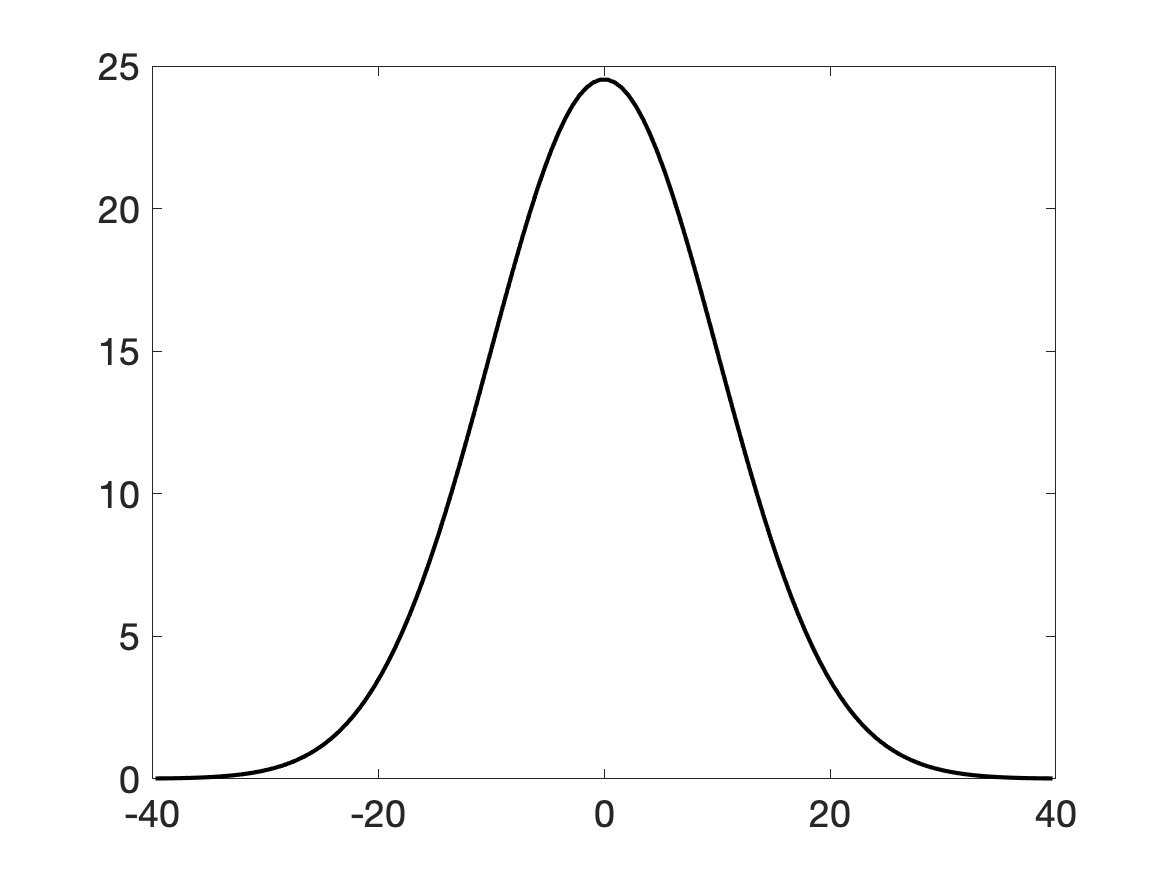}
 &\includegraphics[width=0.35\textwidth]{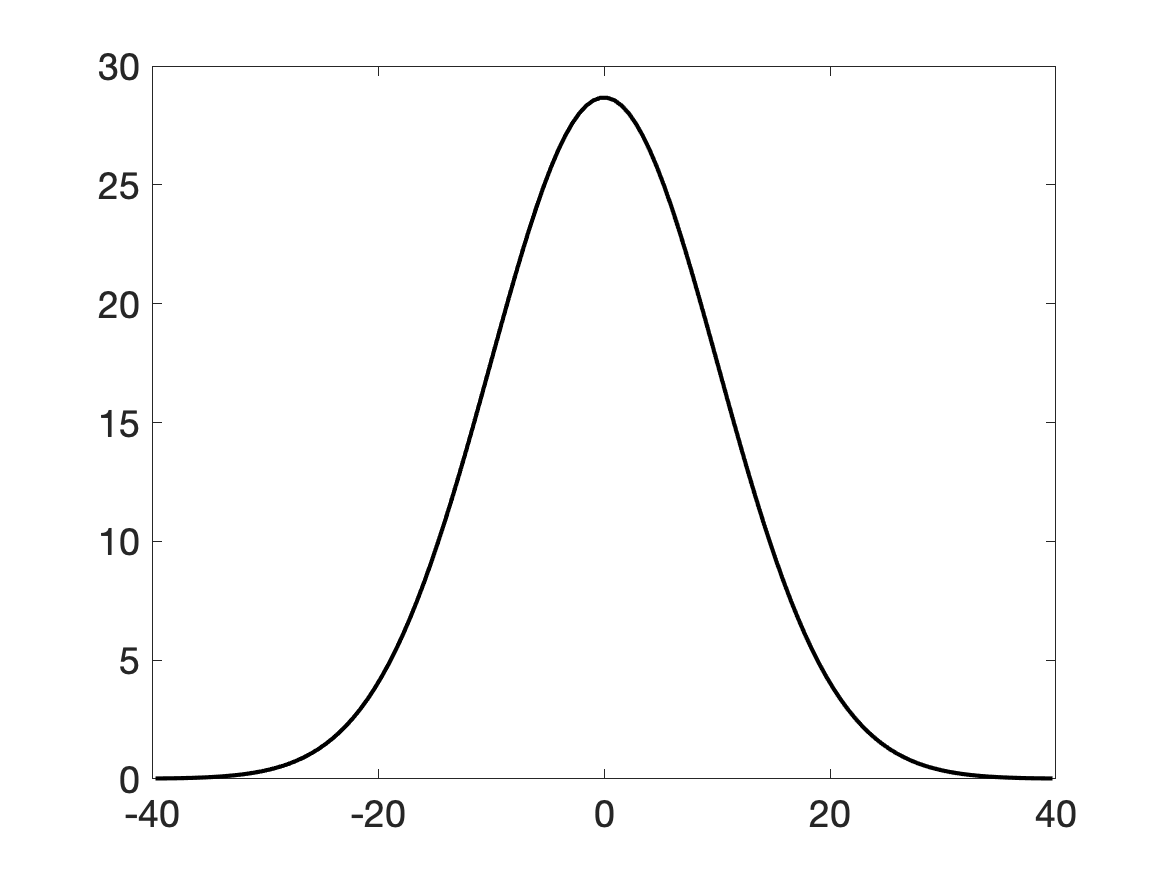}\\
\includegraphics[width=0.35\textwidth]{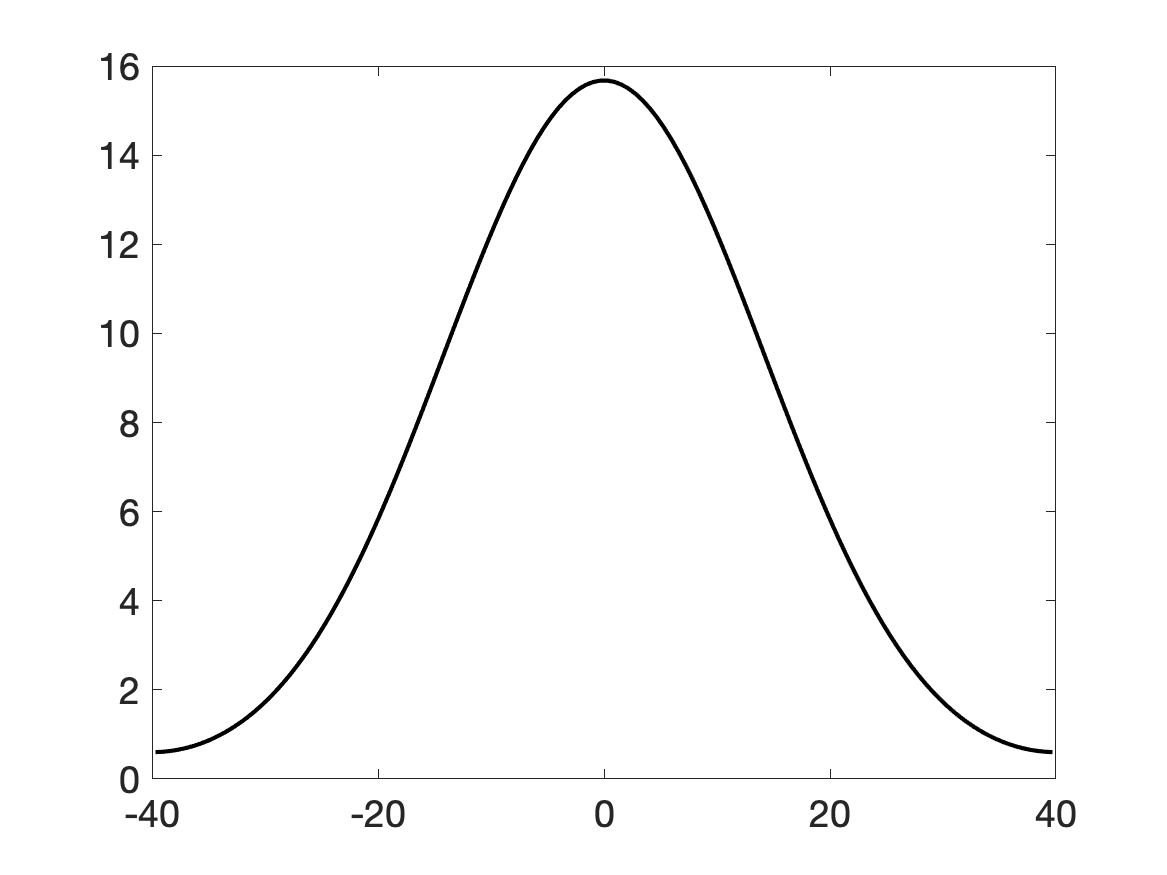}
 &\includegraphics[width=0.35\textwidth]{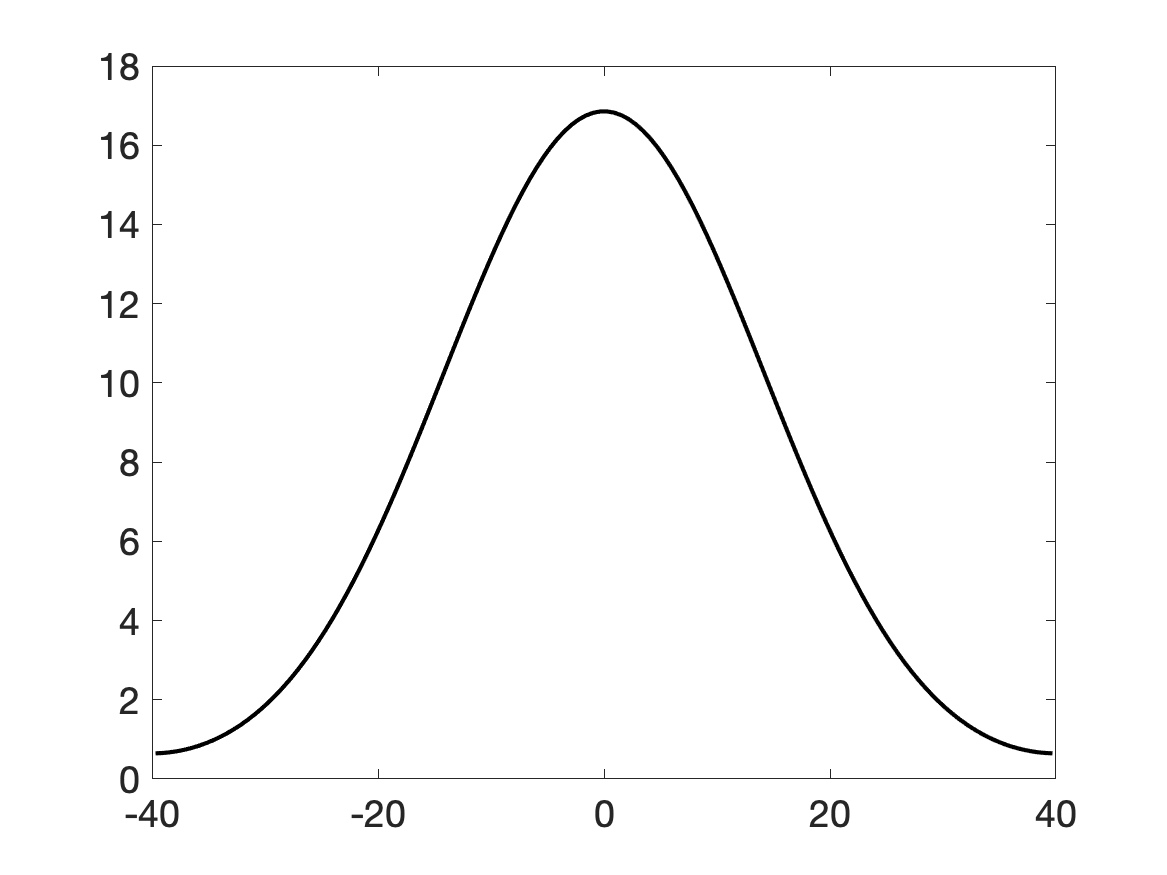}\\
Density on the road & Density on the road\\
\end{tabular}\end{center}
\caption{Profiles of $u$, density on the road, for Test Case 1 (on the left) and Test Case 2 (on the right) at different times : $T=1$, $T=10$, $T=50$, $T=100$ (from the top to the bottom).}
\label{fig:CT12_route}
\end{figure}

\begin{figure}[!htb]
\begin{center}
\begin{tabular}{cc}
 \includegraphics[width=0.45\textwidth]{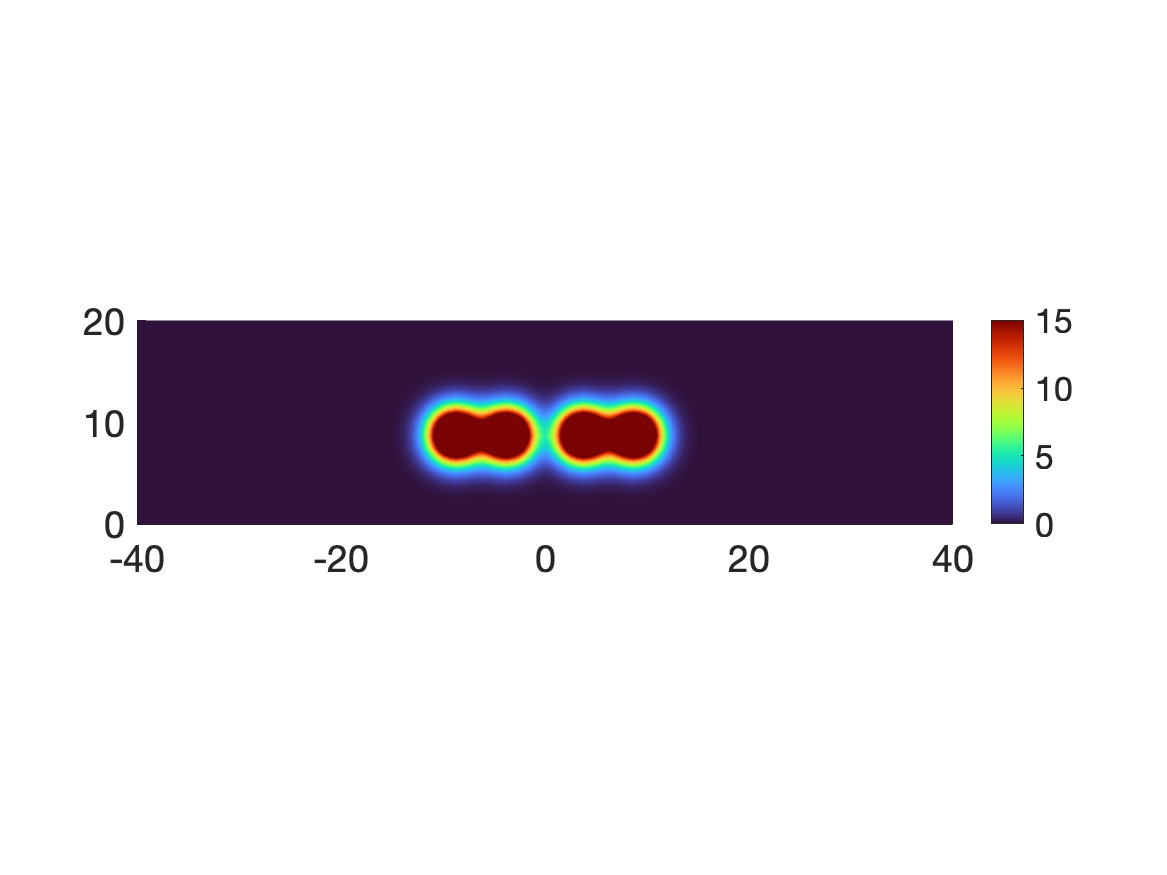}
 &\includegraphics[width=0.45\textwidth]{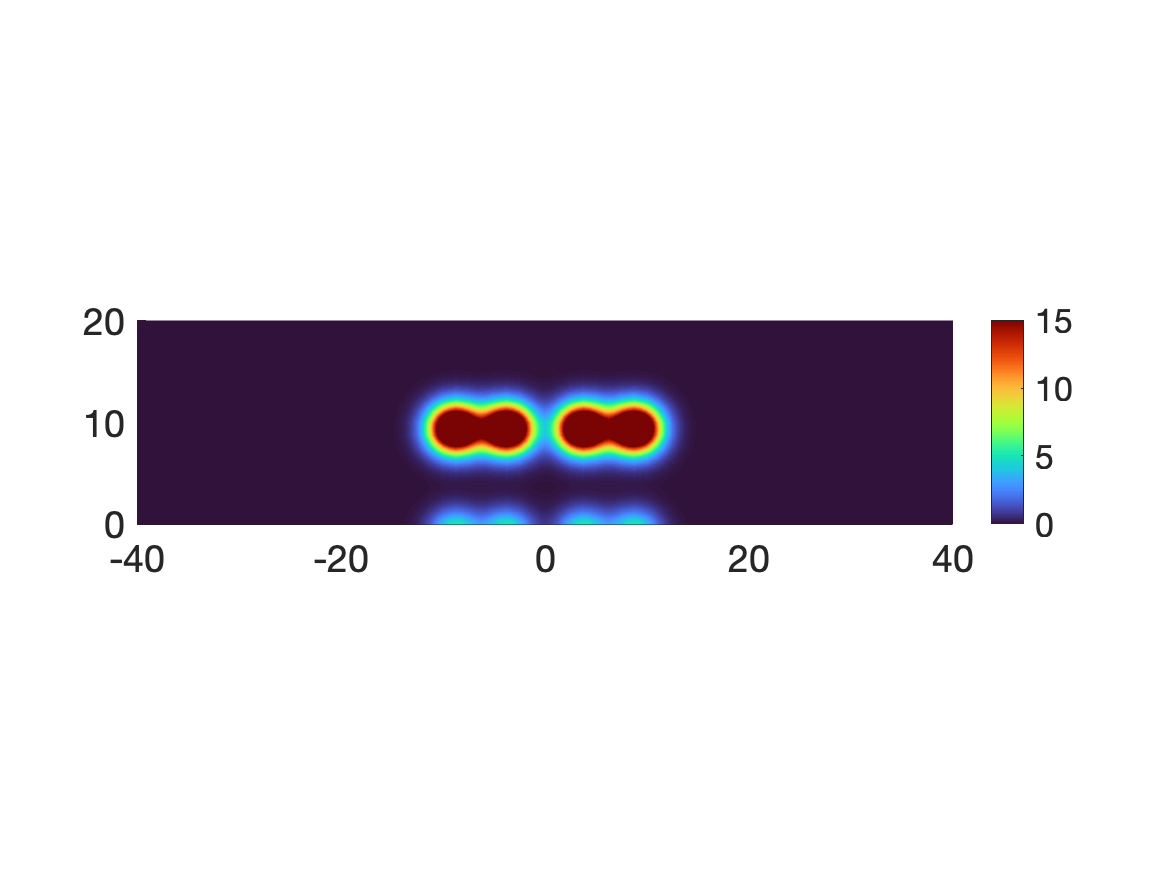}\\[-1.8cm]
 \includegraphics[width=0.45\textwidth]{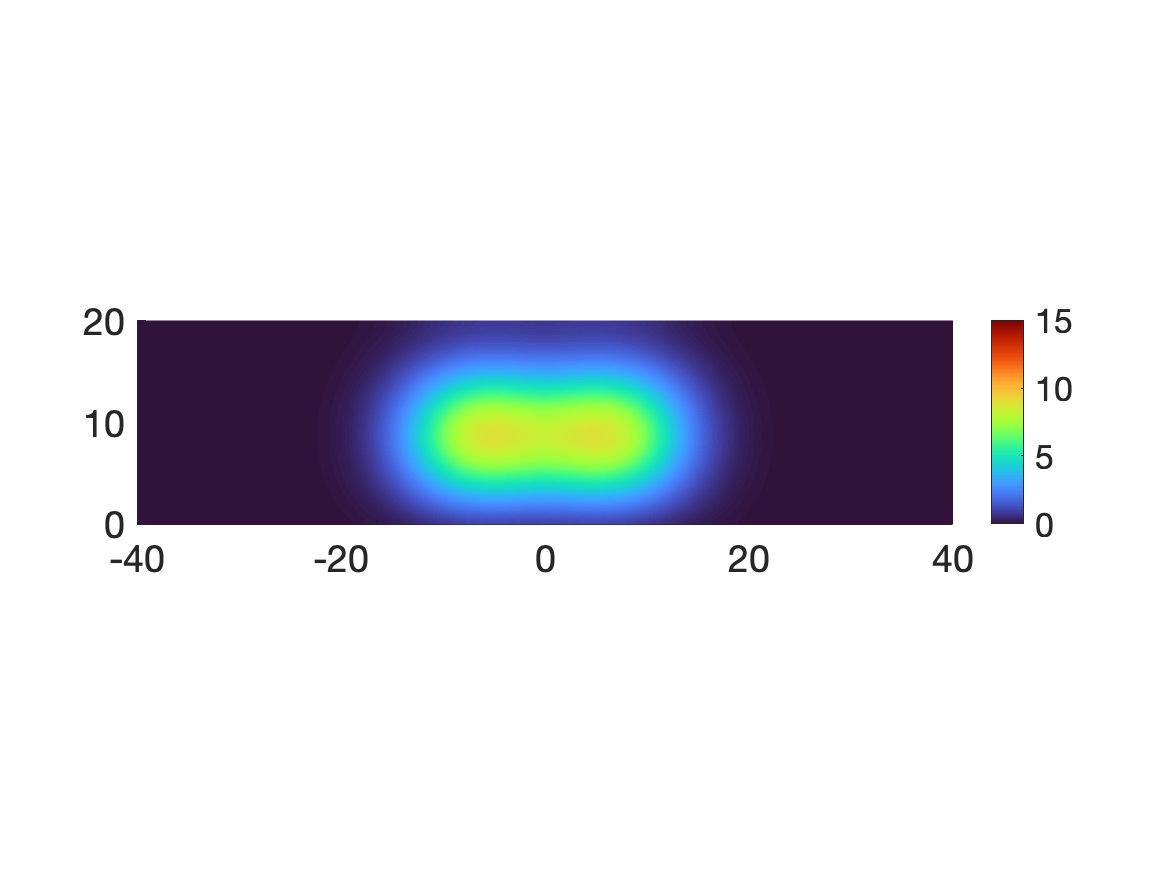}
 &\includegraphics[width=0.45\textwidth]{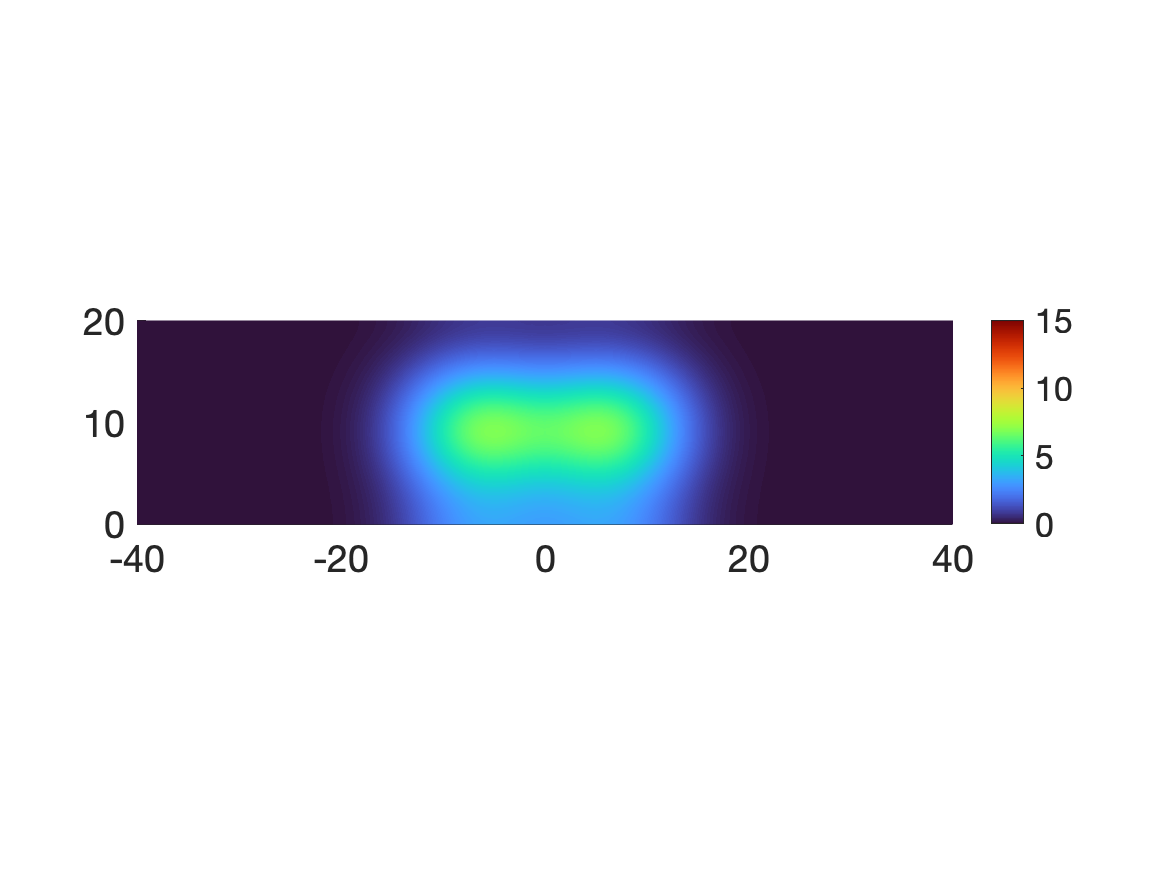}\\[-1.8cm]
\includegraphics[width=0.45\textwidth]{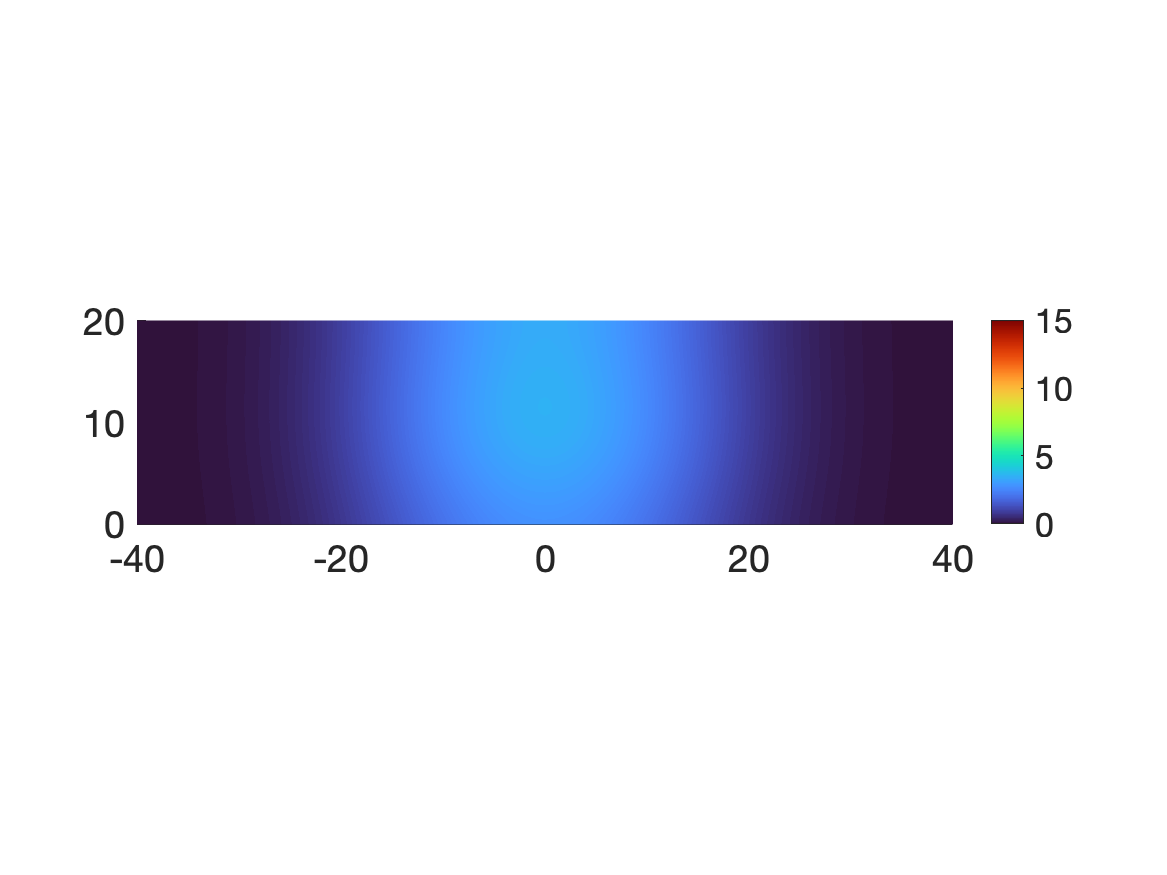}
 &\includegraphics[width=0.45\textwidth]{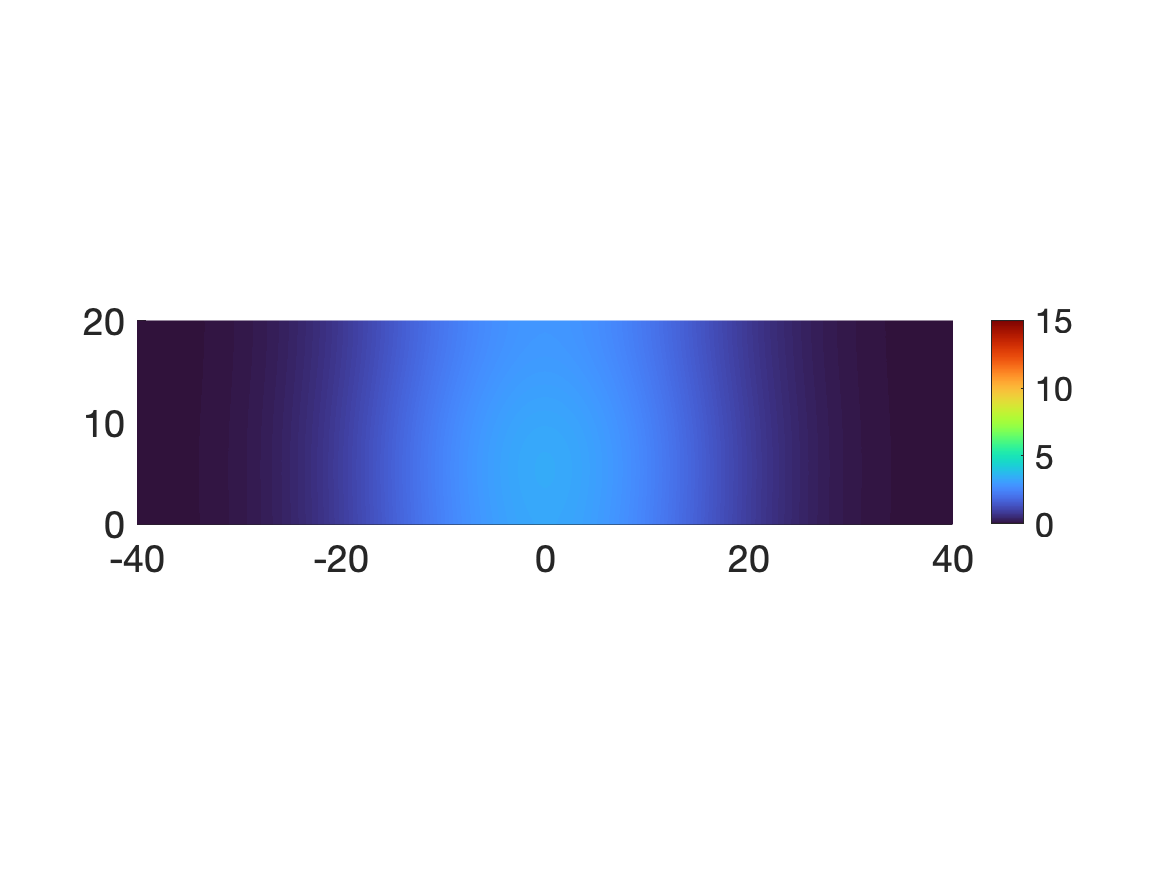}\\[-1.8cm]
\includegraphics[width=0.45\textwidth]{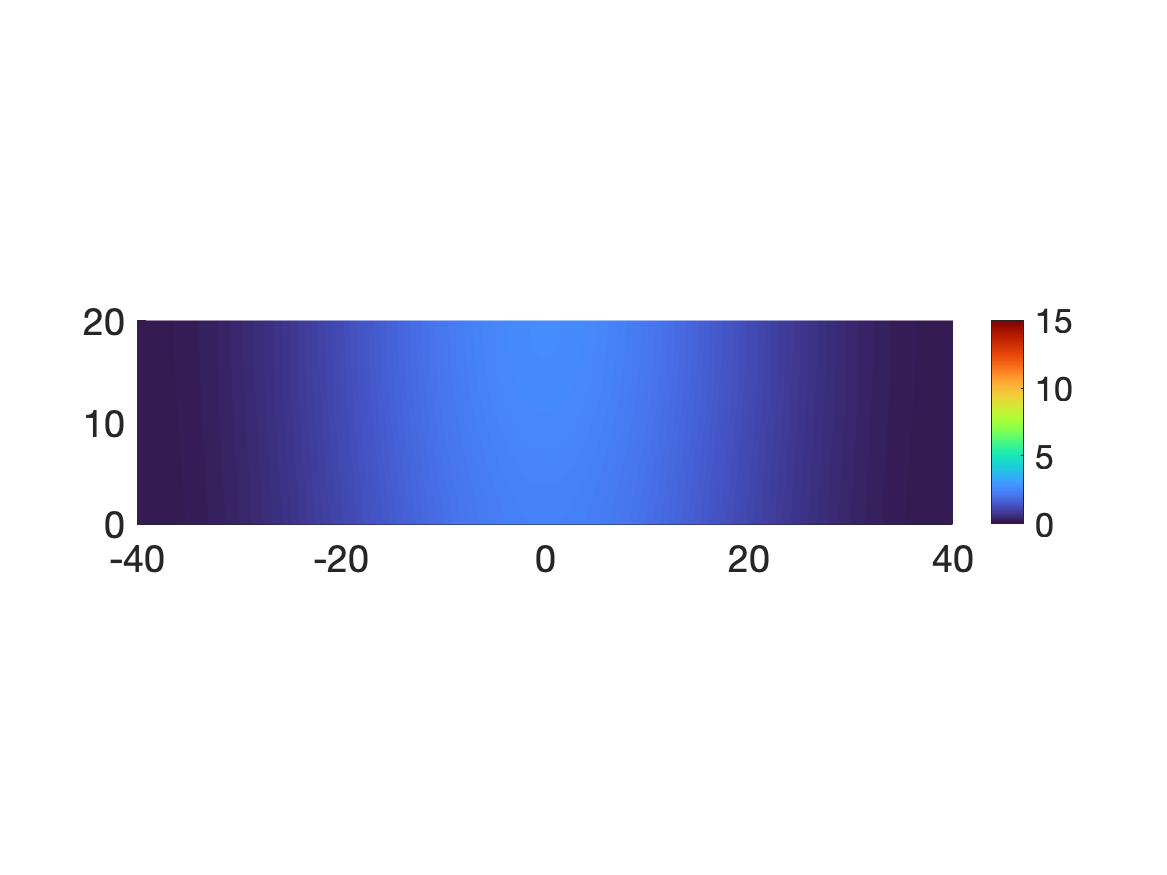}
 &\includegraphics[width=0.45\textwidth]{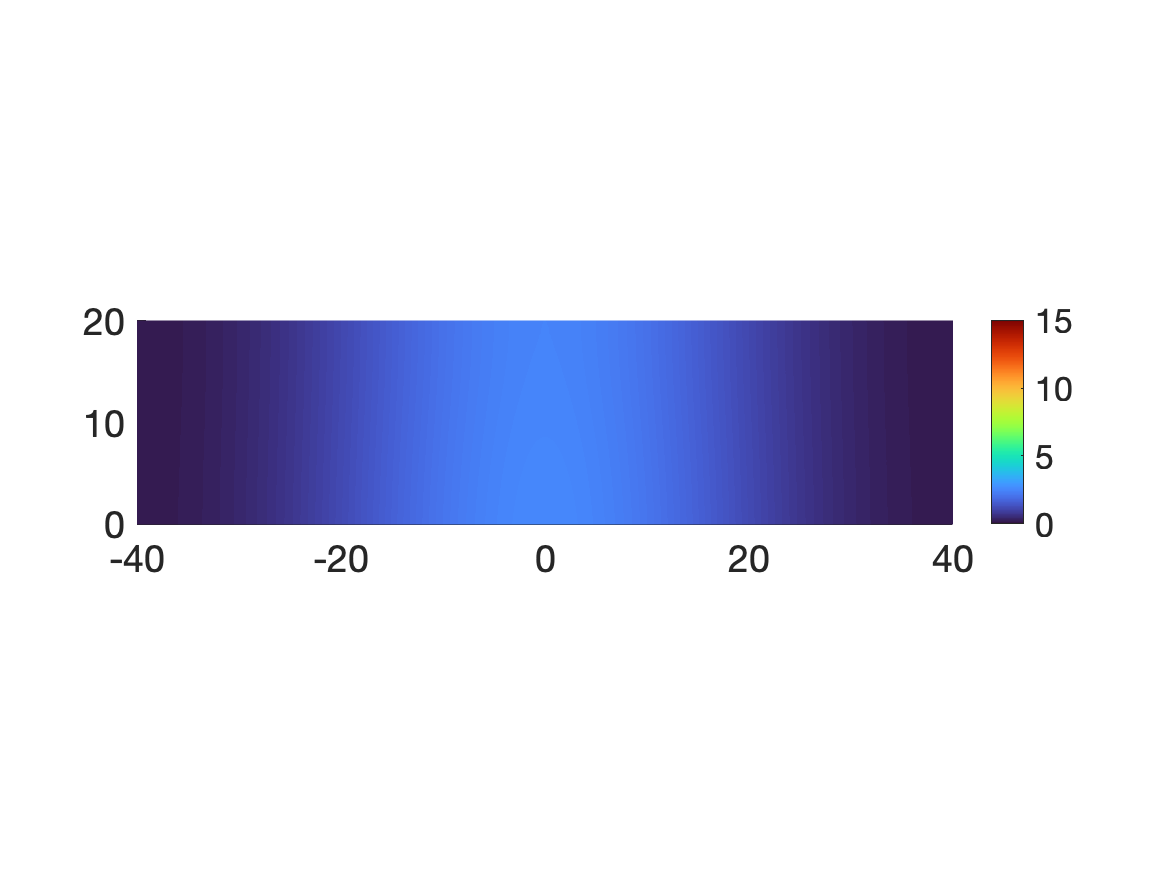}\\[-2.cm]
\end{tabular}\end{center}
\caption{Profiles of $v$, density in the field, for Test Case 3 (on the left) and Test Case 4 (on the right) at different times : $T=1$, $T=10$, $T=50$, $T=100$ (from the top to the bottom).}
\label{fig:CT34_champ}
\end{figure}

\begin{figure}[!htb]
\begin{center}
\begin{tabular}{cc}
 \includegraphics[width=0.35\textwidth]{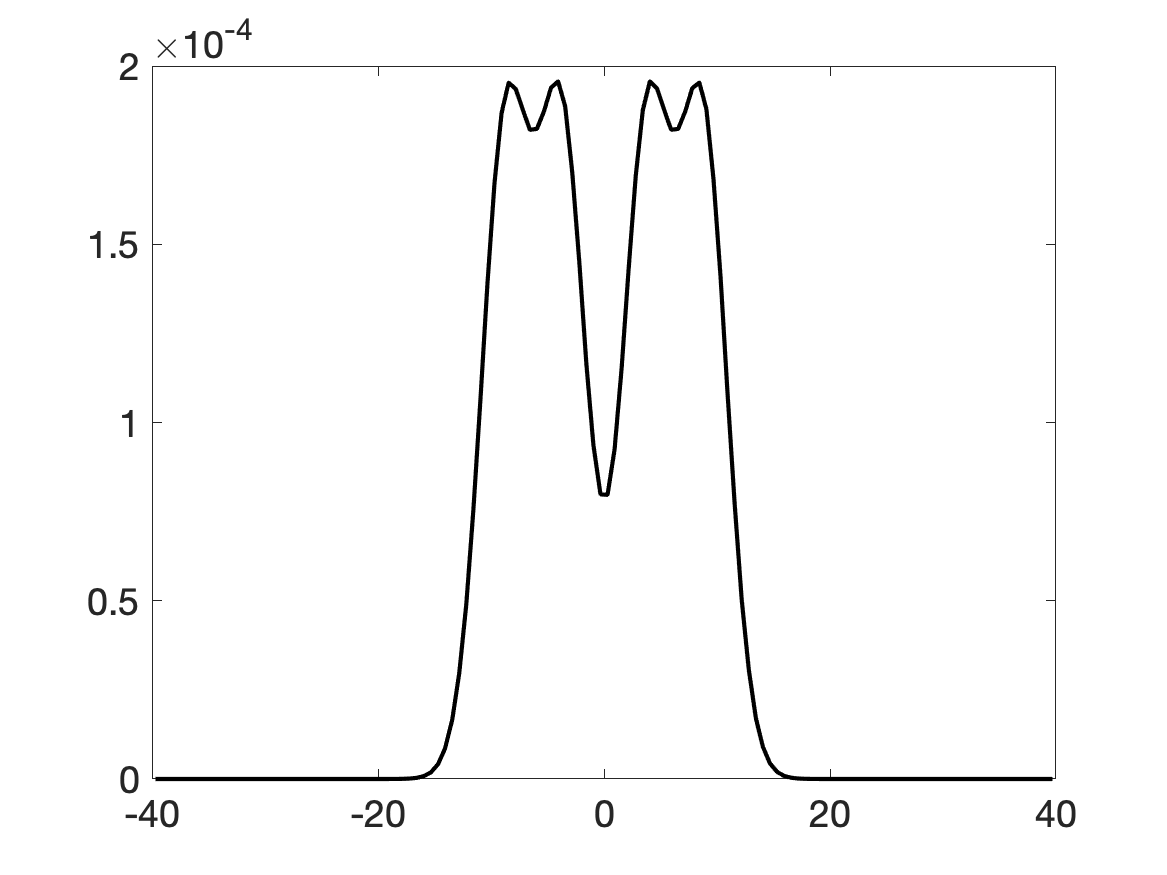}
 &\includegraphics[width=0.35\textwidth]{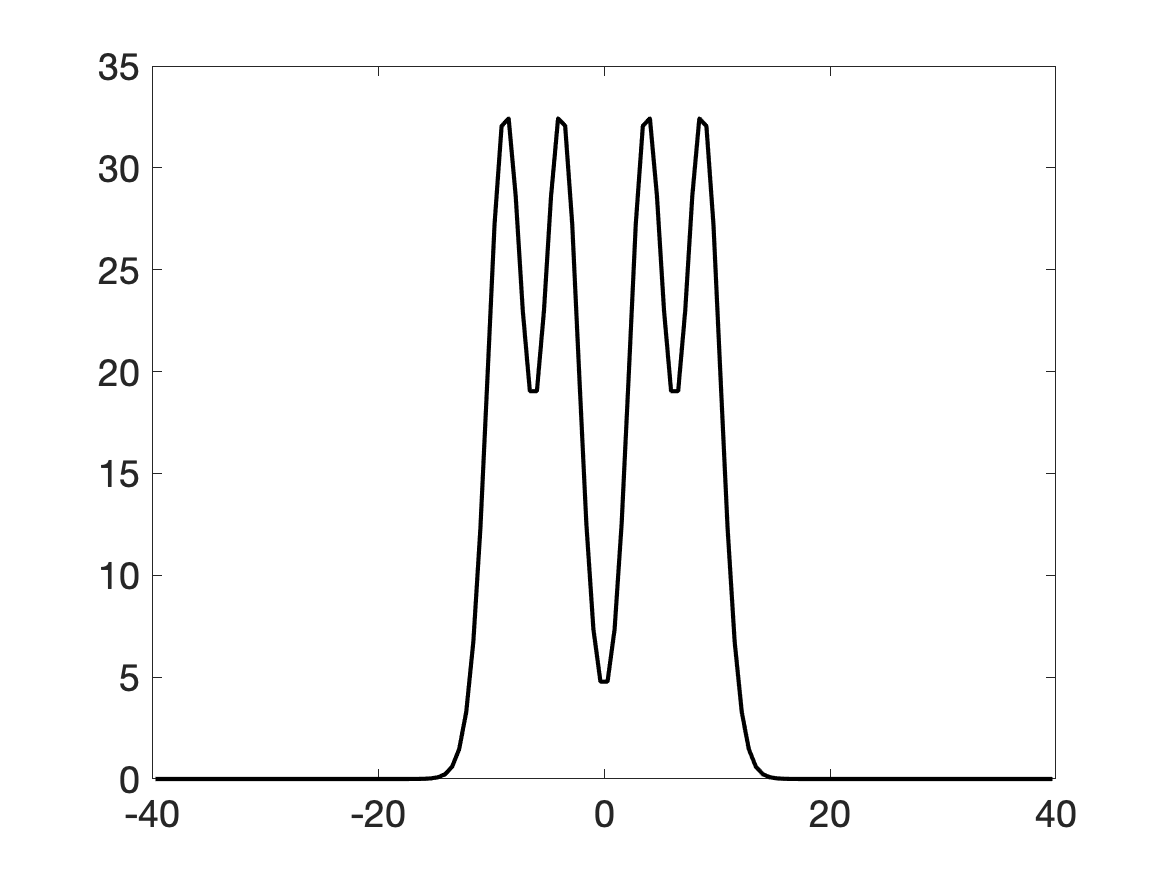}\\
 \includegraphics[width=0.35\textwidth]{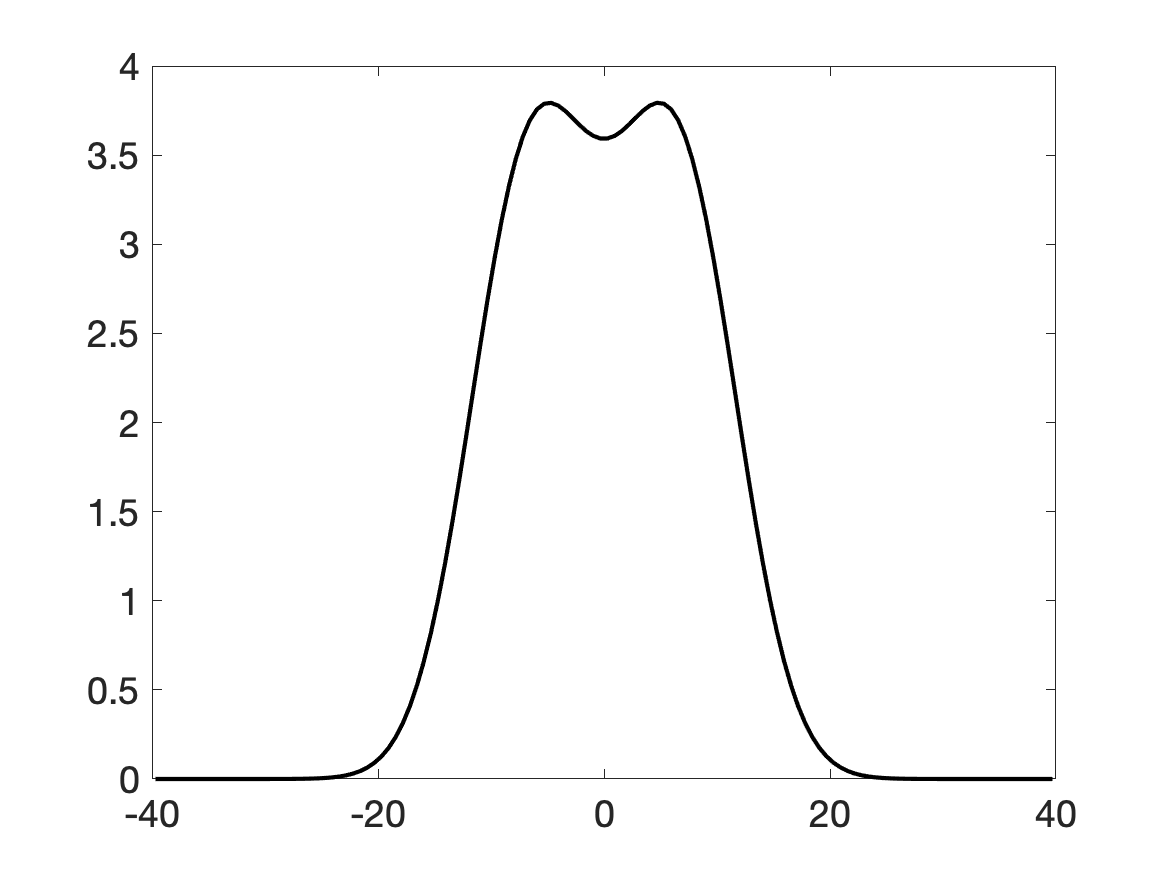}
 &\includegraphics[width=0.35\textwidth]{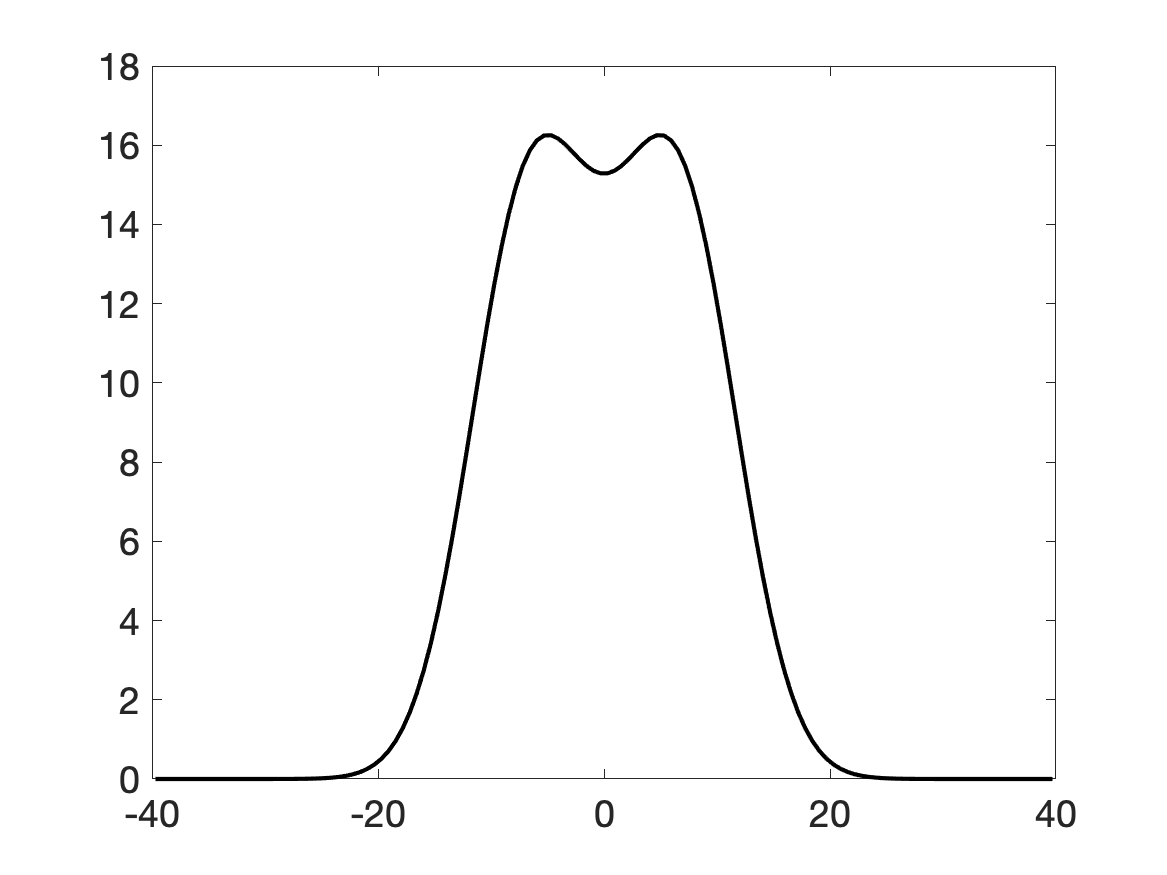}\\
\includegraphics[width=0.35\textwidth]{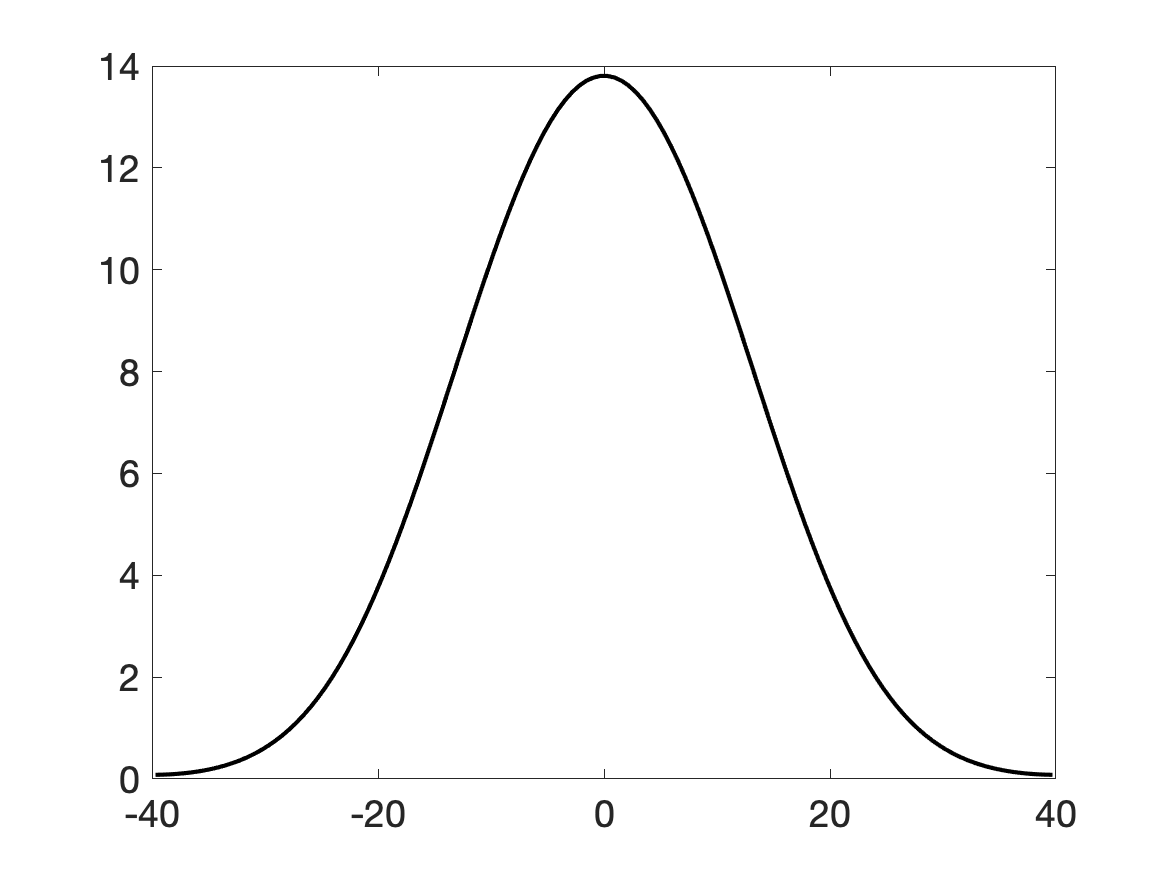}
 &\includegraphics[width=0.35\textwidth]{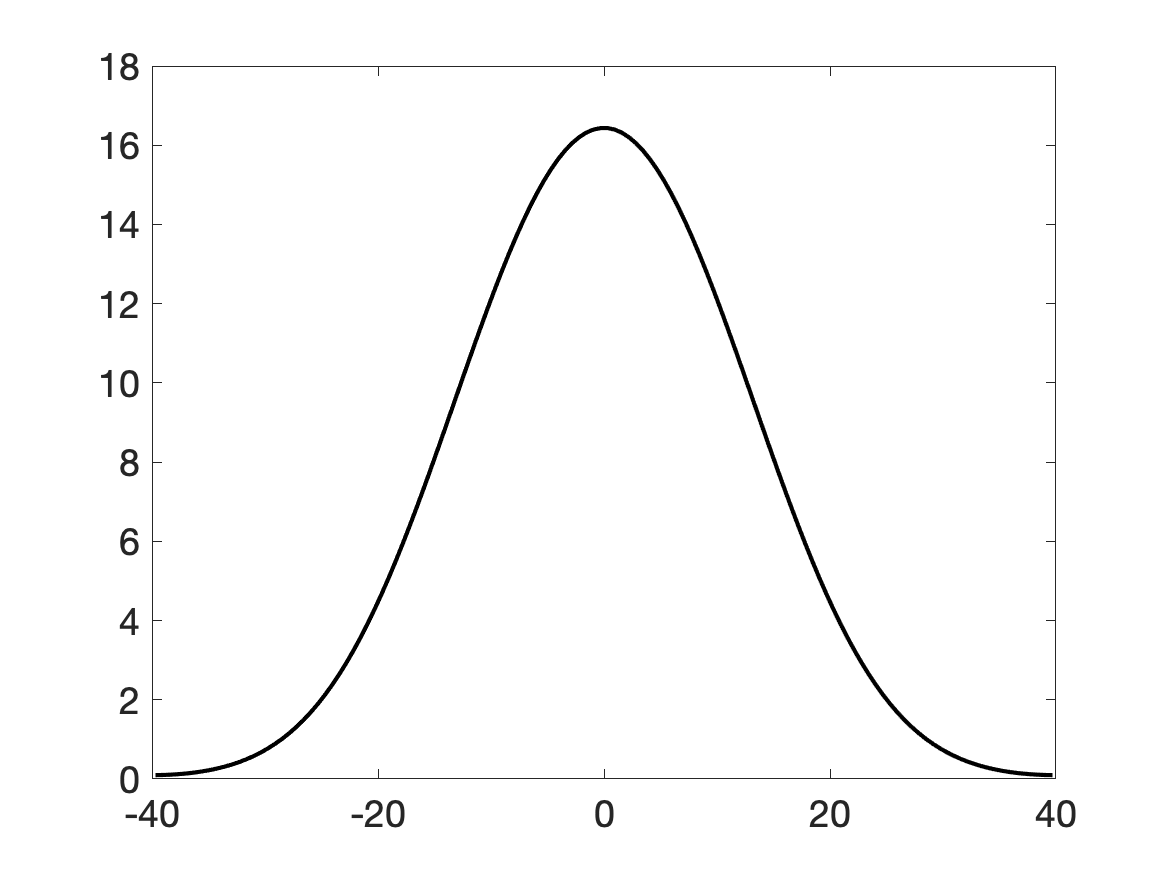}\\
\includegraphics[width=0.35\textwidth]{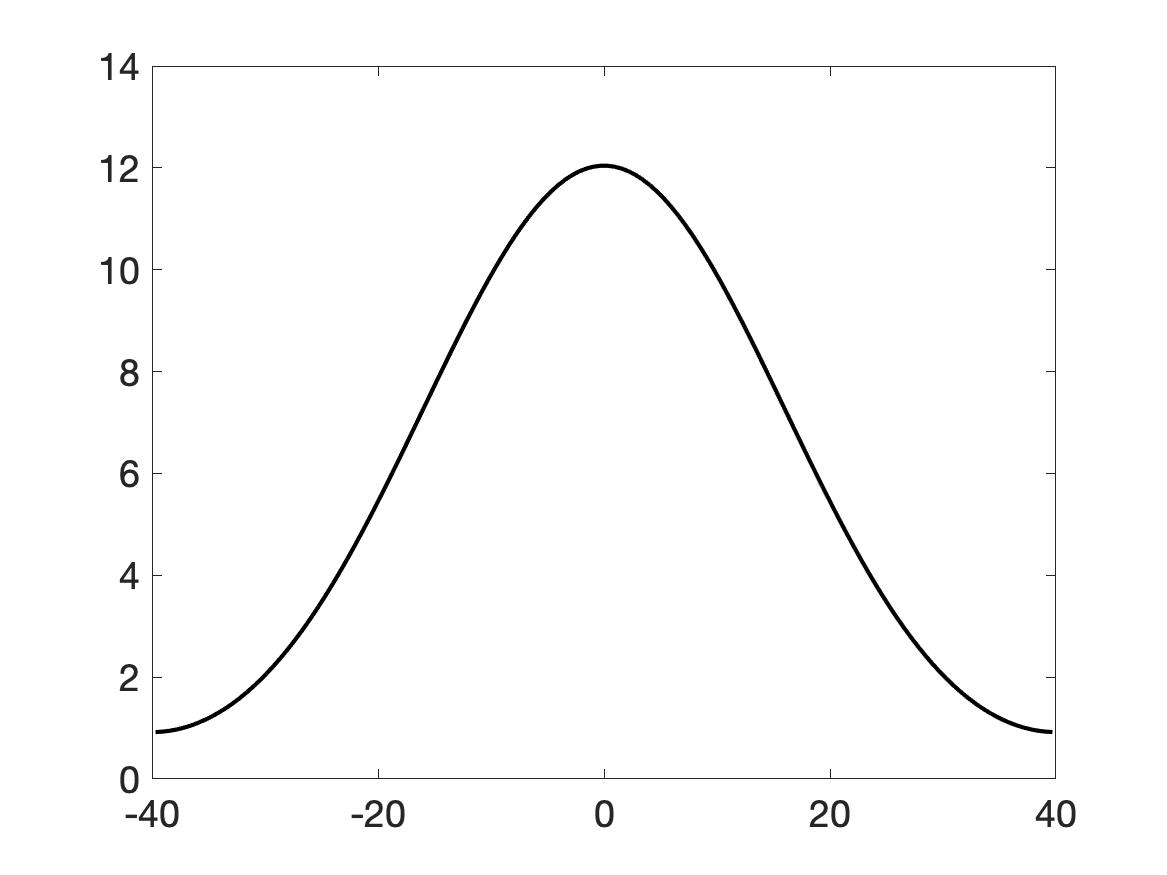}
 &\includegraphics[width=0.35\textwidth]{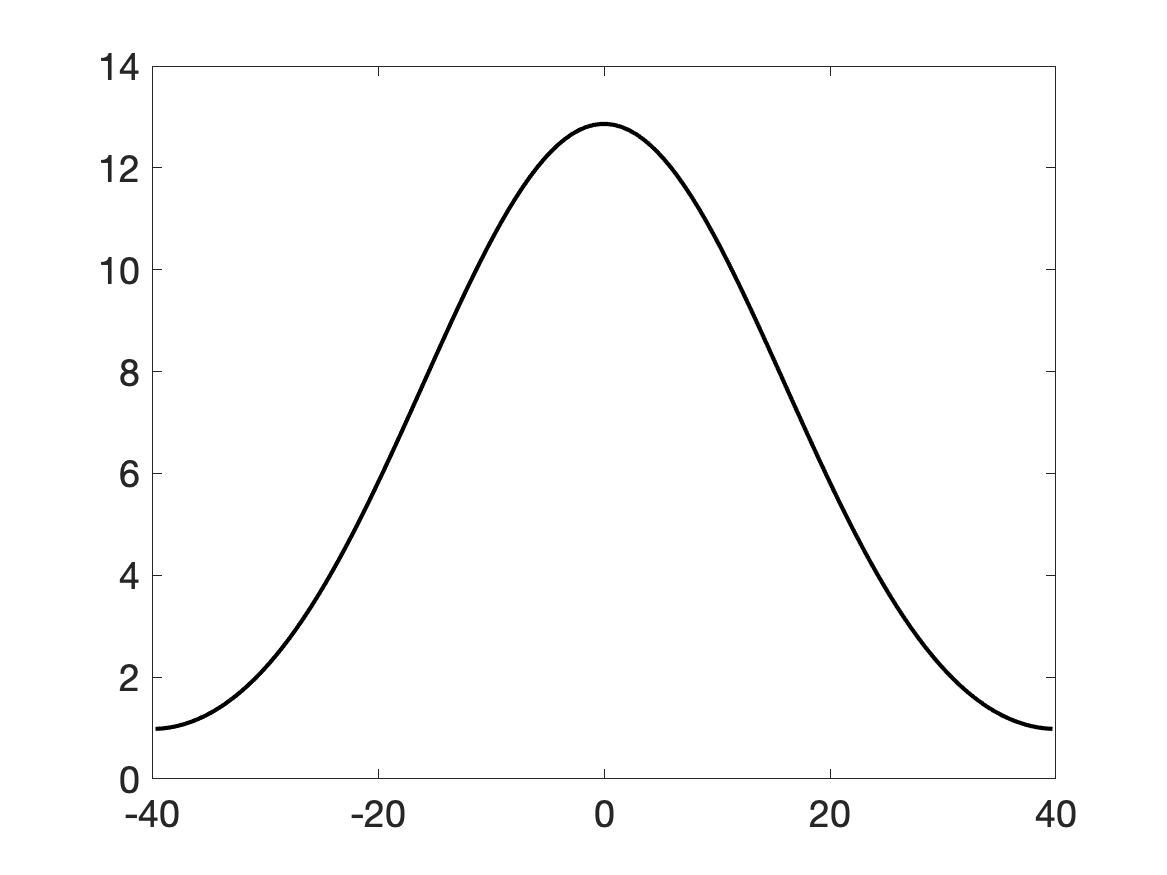}\\
Density on the road & Density on the road\\
\end{tabular}\end{center}
\caption{Profiles of $u$, density on the road, for Test Case 3 (on the left) and Test Case 4 (on the right) at different times : $T=1$, $T=10$, $T=50$, $T=100$ (from the top to the bottom).}
\label{fig:CT34_route}
\end{figure}


\begin{figure}[!htb]
\begin{center}
\includegraphics[width=0.5\textwidth]{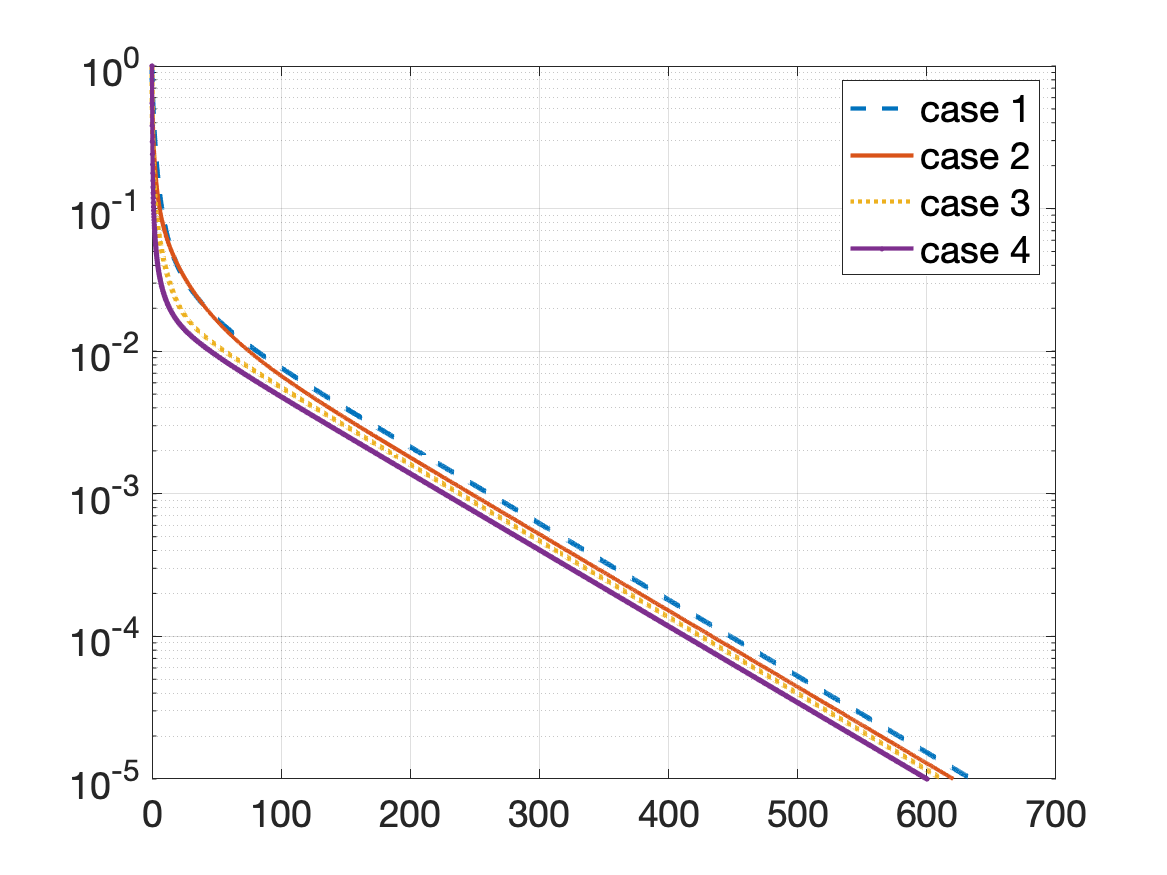}
\end{center}
\caption{Evolution in time of the relative entropies (divided by the value at the first time step) for the four test cases.}\label{fig:CT1to4_entropies}
\end{figure}

\subsection{Entropy decay rate as a function of the different parameters}\label{ss:entropy-decay-rate}

In the light of the four test cases in subsection \ref{ss:test}, we believe that the founding population, in particular its location and fragmentation, has an effect on the behavior for small times but not on the asymptotic decay rate. Therefore to study the latter, we  now restrict ourselves  to Test Case 1. 

We fix $\mu=1$ and $\nu=5$ as before. Next, we fix $d=1$, respectively $D=1$, and compute the decay rate $\Lambda_{num}$ as a function of $D>0$, respectively $d>0$.  The results are shown in figure \ref{fig:D-d}. They are obtained with a time step of $\Delta t = 10^{-1}$ and with a mesh of 3584, respectively  896, triangles for the dependence on $D$, respectively $d$. As the decay rate tends to $0$, we need a huge number of time steps to compute a relevant value, and the coarser the mesh, the faster it goes.

First we note that, as expected, the decay rate is increasing w.r.t. both $D>0$ and $d>0$. Also, we observe that $\Lambda_{num}(d=1,D)\in (9.50\cdot 10^{-3}, 2.51 \cdot 10^{-2})$ while $\Lambda_{num}(d,D=1)\in (0, 2.36)$. In view of these ranges of values taken by $\Lambda_{num}(d,D)$, $d$ seems to have a larger influence on the decay rate than $D$.

Next, we compare the numerically computed $\Lambda_{num}=\Lambda_{num}(d,D)$ with the $\Lambda_2=\Lambda_2(d,D)$ provided by \eqref{Lambda-2}. Obviously,  \eqref{Lambda-2} is recast
$$
\Lambda_2(d,D)=\left\{
\begin{array}{ll}
\min\left(c_1; c_2 D\right) &  \mbox{ if } d=1,\vspace{3pt} \\
\min\left(c_3 d; c_4\right)  &  \mbox{ if } D=1,\end{array}
\right.
$$
for some $c_i>0$ ($1\leq i \leq 4$).

First, we consider the case of large diffusion coefficients. Both 
 $\Lambda_2(1,+\infty)$ and $\Lambda_2(+\infty,1)$ are positive constants, and so are $\Lambda_{num}(1,+\infty)$ and $\Lambda_{num}(+\infty,1)$, which is qualitatively satisfactory. 

Next, we consider the case of small diffusion coefficients.
We observe that $\Lambda_{num}(d,1)\to 0$ as $d\to 0$, which was already expected since $\Lambda_2(d,D)\to 0$ as $d\to 0$.  The reason is that, if $0<d \ll 1$, individuals will very slowly invade the whole field (in particular the zone far from the road). More interestingly, we observe that $\Lambda_{num}(1,D)$ tends to a nonzero value as $D\to 0$. This can be understood as follows: even if $0<D\ll 1$, individuals will be able to invade the whole road via a combination of \lq\lq invasion of the field'' and \lq\lq exchange terms''. However, notice that $\Lambda_2(d,D)\to 0$ as $D\to 0$, revealing that our analysis is far from optimal in the regime $0<D\ll 1$. 

This sheds light on the (different) roles of $d$ and $D$. Also the singular limit problem $D\to 0$ would deserve further investigations.

\begin{figure}[!htb]
\begin{center}
\begin{tabular}{cccc}
\rotatebox{90}{\hspace{2.25cm}$\Lambda_{num}$}\hspace*{-0.5cm}&\includegraphics[width=0.45\textwidth]{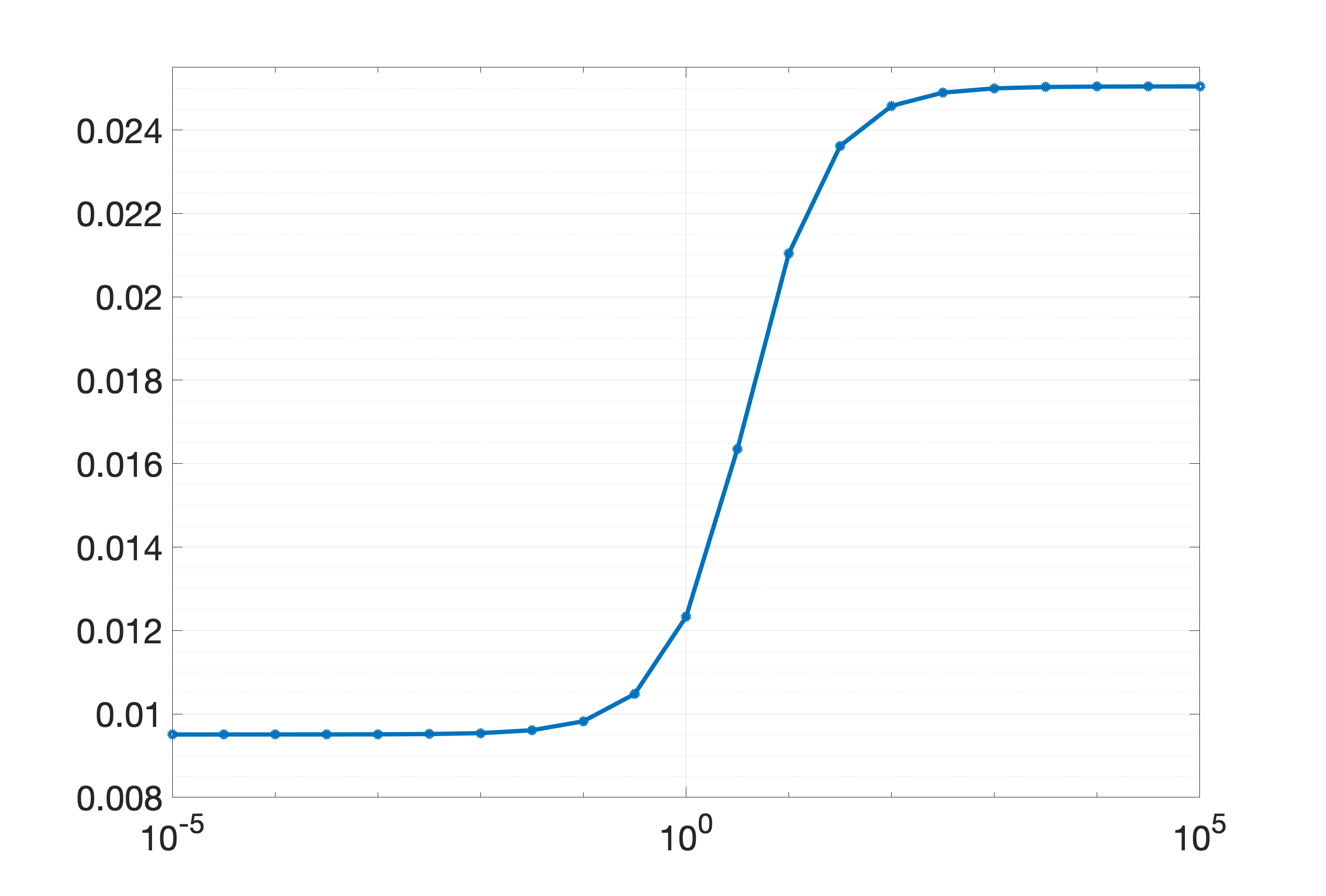}&\rotatebox{90}{\hspace{2.25cm}$\Lambda_{num}$}\hspace*{-0.5cm}&
\includegraphics[width=0.44\textwidth]{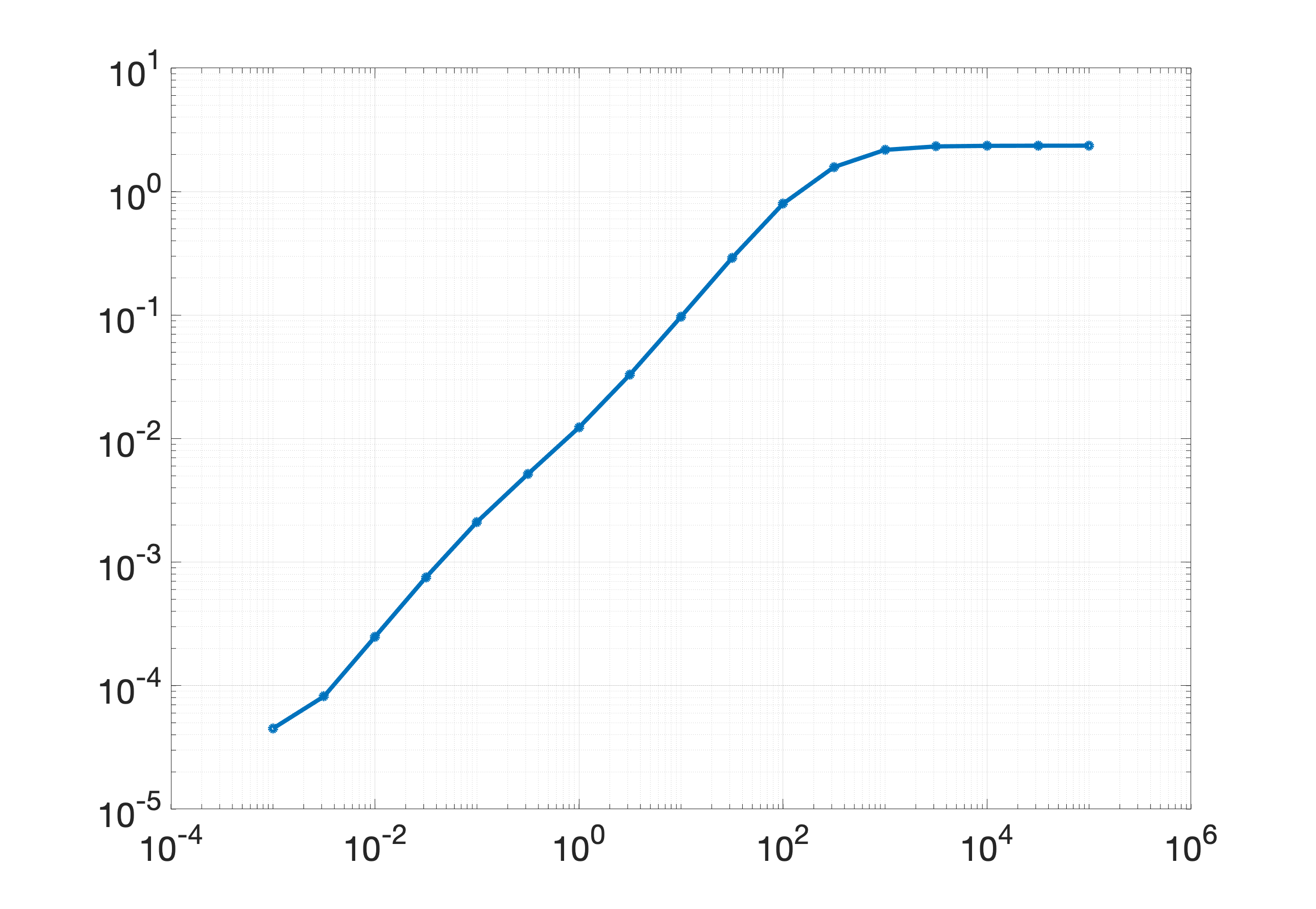}\\[-0.3cm]
& $D$ && $d$
\end{tabular}
\end{center}
\caption{Computed decay rate $\Lambda_{num}$  as a function of $D$ (on the left, in log-lin scale) and as a function of $d$ (on the right, in log-log scale).}\label{fig:D-d}
\end{figure}

\medskip

\noindent{\bf Acknowledgement.} M. A. is supported by  the {\it région Normandie} project BIOMA-NORMAN 21E04343 and the ANR project DEEV ANR-20-CE40-0011-01.  C. C.-H. acknowledges support from the Labex CEMPI (ANR-11-LABX-0007-01).

\bibliographystyle{siam}  

\bibliography{biblio}

\end{document}